\documentclass{cimart}

\usepackage{enumitem}

\DeclareMathOperator{\Imag}{Im}
\DeclareMathOperator{\Ker}{Ker}
\DeclareMathOperator{\Sym}{Sym}

\theoremstyle{cmdefinition}
\newtheorem{notation}{Notation}

\title{Orthogonal idempotents in symmetric tensor powers of composition algebras}

\authors{Aharon Razon}

\authorinfo{Elta Systems Ltd., Israel}{razona@elta.co.il}

\abstract{%
    We explicitly find a complete set of $\frac{1}{4}(n+2)^2$ (resp.\ $\frac{1}{4}(n+1)(n+3)$)
    primitive orthogonal idempotents in $\Sym^n\mathbb{H}\otimes_\mathbb{R}\mathbb{C}$ if $n$ is even (resp.\ odd),
    where $\Sym^n\mathbb{H}$ is the $n$-th symmetric power of the Hamilton quaternion algebra $\mathbb{H}$.
    We also give a complete set of $\frac{1}{4}(n+2)^2$ (resp.\ $\frac{1}{8}(n+1)(n+3)$)
    primitive orthogonal idempotents in $\Sym^n\mathbb{H}$ if $n$ is even (resp.\ odd).
    Moreover, we explicitly find a complete set of $\frac{1}{24}(n+2)(n+3)(n+4)$ (resp.\ $\frac{1}{24}(n+1)(n+3)(n+5)$)
    primitive orthogonal idempotents in the associative subalgebra 
    $\big(\Sym^n\mathbb{H}\cdot Z(\Sym^n\mathbb{O})\big)\otimes_\mathbb{R}\mathbb{C}$
    of $\Sym^n\mathbb{O}\otimes_\mathbb{R}\mathbb{C}$ if $n$ is even (resp.\ odd),
    where $\Sym^n\mathbb{O}$ is the $n$-th symmetric power of the Cayley octonion algebra $\mathbb{O}$
    and $Z(\Sym^n\mathbb{O})$ is its center.
    We also give a complete set of $\frac{1}{24}(n+2)(n+3)(n+4)$ (resp.\ $\frac{1}{48}(n+1)(n+3)(n+5)$)
    primitive orthogonal idempotents in the associative subalgebra 
    $\Sym^n\mathbb{H}\cdot Z(\Sym^n\mathbb{O})$
    of $\Sym^n\mathbb{O}$ if $n$ is even (resp.\ odd).
}



\VOLUME{34}
\YEAR{2026}
\ISSUE{1}
\NUMBER{13}
\DOI{https://doi.org/10.46298/cm.18005}
\editinfo{April 14, 2026}{June 17, 2026}{Adam Chapman}
\acknowledgments{The author thanks the anonymous referees for some helpful remarks.}

\begin{document}

\section*{Introduction}

Let $\mathcal{C}$ be either the Hamilton quaternion algebra $\mathbb{H}$ or the Cayley octonion algebra $\mathbb{O}$ 
over the field of real numbers $\mathbb{R}$.
Then $\mathcal{C}$ is a composition algebra of dimension 4 or 8, respectively, over $\mathbb{R}$.
For each positive integer $n$,
$\Sym^n\mathcal{C}$ is a direct sum of central simple algebras over $\mathbb{R}$ 
\cite[Prop.\ 3.6(f), Prop.\ 3.9(b), and Prop.\ 3.12(a)]{Raz24}:
\begin{equation}
\label{eq:decomp}
\Sym^n\mathcal{C}=\bigoplus_{m=\left\lceil \frac{n}{2} \right\rceil}^n\mathbb{S}_m^{(n)}\mathcal{C}
\end{equation}
such that $\mathbb{S}_m^{(n)}\mathcal{C}\cong T_{2m-n}\mathcal{C}:=\mathbb{S}_{2m-n}^{(2m-n)}\mathcal{C}$ 
for each integer $m$ between $\left\lceil \frac{n}{2} \right\rceil$ and $n$.
Here $\Sym^n\mathcal{C}$ is the ring of symmetric tensors of $\mathcal{C}^{\otimes n}$, 
i.e. the fixed ring of $\mathcal{C}^{\otimes n}$ under the action of the symmetric group $\mathfrak{S}_n$ defined by
$\sigma(x_1\otimes\cdots\otimes x_n)=x_{\sigma^{-1}(1)}\otimes\cdots\otimes x_{\sigma^{-1}(n)}$ for $\sigma\in\mathfrak{S}_n$ and
$x_1,\dots,x_n\in\mathcal{C}$.
For $x\in\mathcal{C}^{\otimes n}$ we denote $x^\vee:=\frac{1}{n!}\sum_{\sigma\in\mathfrak{S}_n}\sigma(x)$.

Section 3 of \cite{Raz24} gives an explicit construction of an orthogonal system of idempotents:
for each integer $m$ between $\left\lceil \frac{n}{2} \right\rceil$ and $n$, a central idempotent 
$e_{m,\mathcal{C}}^{(n)}\in\Sym^n\mathcal{C}$ is constructed
such that $\mathbb{S}_m^{(n)}\mathcal{C}=\Sym^n\mathcal{C}\cdot e_{m,\mathcal{C}}^{(n)}$, 
$e_{m,\mathcal{C}}^{(n)}e_{m',\mathcal{C}}^{(n)}=0$ if $m\ne m'$, and
$\sum_{m=\left\lceil \frac{n}{2} \right\rceil}^ne_{m,\mathcal{C}}^{(n)}=1^{\otimes n}$
\cite[Lemma 3.4 and Prop.\ 3.6(c)]{Raz24}.
In particular, $e_{m,\mathcal{C}}^{(n)}$ commutes and associates with all elements of $\Sym^n\mathcal{C}$.

In this work we continue the work done in \cite{Raz21,  Raz24, Raz25}
and study further the structure of the simple components
$\mathbb{S}_m^{(n)}\mathcal{C}\cong T_{2m-n}\mathcal{C}$ of the decomposition \eqref{eq:decomp} of $\Sym^n\mathcal{C}$.
We embed the quaternion algebra $\mathbb{H}$ inside the octonion algebra $\mathbb{O}$, so $\Sym^n\mathbb{H}\subseteq\Sym^n\mathbb{O}$.
The aim of the current study is to give explicitly a complete set of primitive orthogonal idempotents 
(see Section~\ref{I})
in $\Sym^n\mathbb{H}\cdot e_{m,\mathcal{C}}^{(n)}$ and in $\big(\Sym^n\mathbb{H}\otimes_\mathbb{R}\mathbb{C}\big)\cdot e_{m,\mathcal{C}}^{(n)}$,
where $\mathbb{C}$ is the field of complex numbers.
Note that $\Sym^n\mathbb{H}\cdot e_{m,\mathbb{H}}^{(n)}=\mathbb{S}^{(n)}_m\mathbb{H}$
and $\Sym^n\mathbb{H}\cdot e_{m,\mathbb{O}}^{(n)}$ is an associative subalgebra of $\Sym^n\mathbb{O}$.
Using \eqref{eq:decomp}, this will give a complete set of primitive orthogonal idempotents in 
$\bigoplus_{m=\left\lceil \frac{n}{2} \right\rceil}^n\Sym^n\mathbb{H}\cdot e_{m,\mathcal{C}}^{(n)}$ and in
$\bigoplus_{m=\left\lceil \frac{n}{2} \right\rceil}^n\big(\Sym^n\mathbb{H}\otimes_\mathbb{R}\mathbb{C}\big)\cdot e_{m,\mathcal{C}}^{(n)}$.

Section 6 of \cite{Raz24} studies the case $\mathcal{C}=\mathbb{H}$, where we have by the Wedderburn--Artin theorem
that the central simple algebra $T_n\mathbb{H}$ is isomorphic
either to a matrix algebra over $\mathbb{R}$ or to a matrix algebra over $\mathbb{H}$:
$T_n\mathbb{H} \cong M_{n+1}(\mathbb{R})$, if $n$ is even,
and $T_n\mathbb{H} \cong M_{\frac{n+1}{2}}(\mathbb{H})$, if $n$ is odd \cite[Thm.\ 6.2]{Raz24}.
In the ring of $n\times n$ matrices $M_n(D)$ over a division ring $D$,
the standard diagonal matrices $E_{ii}$ (with $1$ at position $(i,i)$ and $0$ elsewhere)
form a complete set of primitive orthogonal idempotents.
In particular, for each integer $\ell$ between $\left\lceil \frac{n}{2} \right\rceil$ and $n$,
the size of a complete set of primitive orthogonal idempotents in 
$\mathbb{S}_\ell^{(n)}\mathbb{H} \cong T_{2\ell-n}\mathbb{H}$ is $2\ell-n+1$ (resp.\ $\frac{2\ell-n+1}{2}$)
if $n$ is even (resp.\ odd). Therefore, the size of a complete set of primitive orthogonal idempotents in 
$\mathbb{S}_\ell^{(n)}\mathbb{H}\otimes_\mathbb{R}\mathbb{C}$ is $2\ell-n+1$.

In order to describe explicitly a complete set of primitive orthogonal idempotents in 
$\mathbb{S}_\ell^{(n)}\mathbb{H}$ and in $\mathbb{S}_\ell^{(n)}\mathbb{H}\otimes_\mathbb{R}\mathbb{C}$,
we shall need some notation.
Let $e_0=1,e_1,e_2,e_3$ be a basis of $\mathbb{H}$ over $\mathbb{R}$ such that
$e_1^2=e_2^2=e_3^2=-1$, $e_1e_2=-e_2e_1=e_3$, $e_2e_3=-e_3e_2=e_1$, $e_3e_1=-e_1e_3=e_2$.
Let $\sqrt{-1}$ be a root of $X^2+1=0$ in $\mathbb{C}$.
We denote
\[
a := \frac{1}{2}(1+\sqrt{-1}e_1)\in\mathbb{H}\otimes_\mathbb{R}\mathbb{C}
\]
and let
\[
a^c := \frac{1}{2}(1-\sqrt{-1}e_1)
\]
be its complex conjugate.
We also denote
\[
\square:=\frac{1}{2}(1\otimes1+e_1\otimes e_1)\in\Sym^2\mathbb{H}.
\]

\begin{theorem}[{Theorem~\ref{Theorem1}}]
\label{TheoremA}
Let $\ell$ be an integer between $\left\lceil \frac{n}{2} \right\rceil$ and $n$.
\begin{enumerate}[label={\rm(\alph*)},leftmargin=*]
\item The set
\[
\left\{
\binom{n}{k}
\big(a^{\otimes k}\otimes (a^c)^{\otimes(n-k)}\big)^\vee
e^{(n)}_{\ell,\mathbb{H}}
\;\middle|\;
k \in \{n-\ell,\ldots,\ell\}
\right\}
\]
is a complete set of primitive orthogonal idempotents in
$\mathbb{S}_\ell^{(n)}\mathbb{H}\otimes_\mathbb{R}\mathbb{C}$.

\item
If $n$ is odd, then
\begin{align*}
\binom{n}{k}
\frac{n-2k}{2^k}
\sum_{i=0}^{\frac{n-1}{2}-k}
\left(-\frac{1}{2}\right)^i
\frac{\binom{n-2k-i}{i}}{n-2k-i}
\big(1^{\otimes(n-2k-2i)}\otimes\square^{\otimes(k+i)}\big)^\vee
\cdot e^{(n)}_{\ell,\mathbb{H}}
\,,
\end{align*}
$k\in\{n-\ell,\ldots,\frac{n-1}{2}\}$,
form a complete set of primitive orthogonal idempotents in $\mathbb{S}_\ell^{(n)}\mathbb{H}$.

\item
If $n$ is even, then
\[
\left(\frac{1}{2}\right)^{\frac{n}{2}}
\binom{n}{\frac{n}{2}}
\big(\square^{\otimes \frac{n}{2}}\big)^\vee
\cdot e_{\ell,\mathbb{H}}^{(n)}.
\]

together with
\begin{align*}
\binom{n}{k}
\frac{\frac{n}{2}-k}{2^k}
\sum_{i=0}^{\left\lfloor \frac{n}{4}-\frac{k}{2} \right\rfloor}
\left(-\frac{1}{4}\right)^i
\frac{\binom{\frac{n}{2}-k-i}{i}}{\frac{n}{2}-k-i}
&\big(\square^{\otimes(k+2i)}\otimes(1\otimes1-\square)^{\otimes(\frac{n}{2}-k-2i)}\big)^\vee\\
&\cdot \frac{1}{2}\big(1^{\otimes n}+\delta\cdot e_2^{\otimes n}\big)\cdot e^{(n)}_{\ell,\mathbb{H}}\,,
\end{align*}
$k\in\{n-\ell,\ldots,\frac{n}{2}-1\}$, $\delta\in\{-1,1\}$,
if $\ell>\frac{n}{2}$,
form a complete set of primitive orthogonal idempotents in $\mathbb{S}_\ell^{(n)}\mathbb{H}$.
\end{enumerate}
\end{theorem}

Using the decomposition \eqref{eq:decomp}, the size of a complete set of primitive orthogonal idempotents 
in $\Sym^n\mathbb{H}\otimes_\mathbb{R}\mathbb{C}$ is
\[
\sum_{\ell=\left\lceil \frac{n}{2} \right\rceil}^n (2\ell-n+1),
\]
which is $\frac{1}{4}(n+2)^2$ (resp.\ $\frac{1}{4}(n+1)(n+3)$)
if $n$ is even (resp.\ odd).
Also, the size of a complete set of primitive orthogonal idempotents in $\Sym^n\mathbb{H}$ is
\[
\sum_{\ell=\frac{n}{2}}^n (2\ell-n+1) = \frac{1}{4}(n+2)^2
\]
(resp.\ 
\[
\sum_{\ell=\frac{n+1}{2}}^n \frac{2\ell-n+1}{2} = \frac{1}{8}(n+1)(n+3))\]
if $n$ is even (resp.\ odd).
The following corollary of Theorem~\ref{TheoremA}
describes explicitly a complete set of primitive orthogonal idempotents in 
$\Sym^n\mathbb{H}$ and in $\Sym^n\mathbb{H}\otimes_\mathbb{R}\mathbb{C}$.

\begin{corollary}[{Corollary~\ref{Corollary2}}]
\label{CorollaryB}
\begin{enumerate}[label={\rm(\alph*)},leftmargin=*]
\item
The set
\[
\left\{
\binom{n}{k}\bigl(a^{\otimes k}\otimes(a^c)^{\otimes(n-k)}\bigr)^\vee e^{(n)}_{\ell,\mathbb{H}}
\;\middle|\;
\begin{aligned}
&k\in\{n-\ell,\ldots,\ell\},\\
&\ell\in\left\{\left\lceil \frac{n}{2} \right\rceil,\ldots,n\right\}
\end{aligned}
\right\}
\]
is a complete set of primitive orthogonal idempotents in $\Sym^n\mathbb{H}\otimes_\mathbb{R}\mathbb{C}$.

\item
If $n$ is odd, then
\begin{align*}
\binom{n}{k}\frac{n-2k}{2^k}
\sum_{i=0}^{\frac{n-1}{2}-k}(-\frac{1}{2})^i
\frac{\binom{n-2k-i}{i}}{n-2k-i}
\big(1^{\otimes(n-2k-2i)}\otimes\square^{\otimes(k+i)}\big)^\vee\cdot e^{(n)}_{\ell,\mathbb{H}}\,,
\end{align*}
$k\in\{n-\ell,\ldots,\frac{n-1}{2}\}$, $\ell\in\left\{\frac{n+1}{2},\ldots,n\right\}$,
form a complete set of primitive orthogonal idempotents in $\Sym^n\mathbb{H}$.

\item
If $n$ is even, then
\[
\left(\frac{1}{2}\right)^{\frac{n}{2}}\binom{n}{\frac{n}{2}}
\big(\square^{\otimes\frac{n}{2}}\big)^\vee\cdot e_{\frac{n}{2},\mathbb{H}}^{(n)}
\]
and
\[
\left(\frac{1}{2}\right)^{\frac{n}{2}}\binom{n}{\frac{n}{2}}
\big(\square^{\otimes\frac{n}{2}}\big)^\vee\cdot e_{\ell,\mathbb{H}}^{(n)}
\]
together with
\begin{align*}
\binom{n}{k}\frac{\frac{n}{2}-k}{2^k}
\sum_{i=0}^{\frac{n}{4}-\lfloor\frac{k}{2}\rfloor}
\left(-\frac{1}{4}\right)^i
\frac{\binom{\frac{n}{2}-k-i}{i}}{\frac{n}{2}-k-i}
&\big(\square^{\otimes(k+2i)}\otimes(1\otimes1-\square)^{\otimes(\frac{n}{2}-k-2i)}\big)^\vee\\
&\cdot\frac{1}{2}\big(1^{\otimes n}+\delta\cdot e_2^{\otimes n}\big)\cdot e^{(n)}_{\ell,\mathbb{H}}\,,
\end{align*}
$k\in\{n-\ell,\ldots,\frac{n}{2}-1\}$, $\delta\in\{-1,1\}$, $\ell\in\left\{\frac{n}{2}+1,\ldots,n\right\}$,
form a complete set of primitive orthogonal idempotents in $\Sym^n\mathbb{H}$.
\end{enumerate}
\end{corollary}

Now we consider the associative subalgebra $\Sym^n\mathbb{H}\cdot e^{(n)}_{m,\mathbb{O}}$ of $\Sym^n\mathbb{O}$
and the associative subalgebra $(\Sym^n\mathbb{H}\otimes_\mathbb{R}\mathbb{C})\cdot e^{(n)}_{m,\mathbb{O}}$ of 
$\Sym^n\mathbb{O}\otimes_\mathbb{R}\mathbb{C}$,
for each integer $m$ between $\left\lceil \frac{n}{2} \right\rceil$ and $n$.
Using the decomposition \eqref{eq:decomp} for $\mathcal{C}=\mathbb{H}$,
\[
\Sym^n\mathbb{H}\cdot e^{(n)}_{m,\mathbb{O}}
=\bigoplus_{\ell=\left\lceil \frac{n}{2} \right\rceil}^n\Sym^n\mathbb{H}\cdot e^{(n)}_{\ell,\mathbb{H}}\cdot e^{(n)}_{m,\mathbb{O}}.
\]
In Corollary~\ref{MainLemma} we show that if $\ell$ is an integer between $\left\lceil \frac{n}{2} \right\rceil$ and $m$, then
$\Sym^n\mathbb{H}\cdot e^{(n)}_{\ell,\mathbb{H}}\cong\Sym^n\mathbb{H}\cdot e^{(n)}_{\ell,\mathbb{H}}\cdot e^{(n)}_{m,\mathbb{O}}$
by $x\mapsto x\cdot e^{(n)}_{m,\mathbb{O}}$.
Moreover, $e^{(n)}_{\ell,\mathbb{H}}\cdot e^{(n)}_{m,\mathbb{O}}=0$ if $\ell>m$.
It follows that the size of a complete set of primitive orthogonal idempotents in 
$(\Sym^n\mathbb{H}\otimes_\mathbb{R}\mathbb{C})\cdot e^{(n)}_{m,\mathbb{O}}$ is
$\sum_{\ell=\left\lceil \frac{n}{2} \right\rceil}^m(2\ell-n+1)$ which is
${\frac{1}{4}}(2m-n+2)^2$ (resp.\ ${\frac{1}{4}}(2m-n+1)(2m-n+3)$) if $n$ is even (resp.\ odd).
Also, the size of a complete set of primitive orthogonal idempotents in
$\Sym^n\mathbb{H}\cdot e^{(n)}_{m,\mathbb{O}}$ is:
\[
    \begin{cases}
        \sum_{\ell=\frac{n}{2}}^m(2\ell-n+1)=\frac{1}{4}(2m-n+2)^2, 
        & \text{if $n$ is even}, \\
        \sum_{\ell=\frac{n+1}{2}}^m \frac{2\ell-n+1}{2}=\frac{1}{8}(2m-n+1)(2m-n+3),
        & \text{if $n$ is odd}.
    \end{cases}
\]
The following theorem describes explicitly a complete set of primitive orthogonal idempotents in 
$(\Sym^n\mathbb{H}\otimes_\mathbb{R}\mathbb{C})\cdot e^{(n)}_{m,\mathbb{O}}$ and in $\Sym^n\mathbb{H}\cdot e^{(n)}_{m,\mathbb{O}}$.

\begin{theorem}[{Theorem~\ref{Theorem3}}]
\label{TheoremC}
Let $m$ be an integer between $\left\lceil \frac{n}{2} \right\rceil$ and $n$.
\begin{enumerate}[label={\rm(\alph*)},leftmargin=*]
\item
The set
\[
\left\{
\binom{n}{k}\bigl(a^{\otimes k}\otimes(a^c)^{\otimes(n-k)}\bigr)^\vee 
e^{(n)}_{\ell,\mathbb{H}}e^{(n)}_{m,\mathbb{O}}
\;\middle|\;
\begin{aligned}
&k\in\{n-\ell,\ldots,\ell\},\\
&\ell\in\left\{\left\lceil \frac{n}{2} \right\rceil,\ldots,m\right\}
\end{aligned}
\right\}
\]
is a complete set of primitive orthogonal idempotents in $(\Sym^n\mathbb{H}\otimes_\mathbb{R}\mathbb{C})\cdot e^{(n)}_{m,\mathbb{O}}$.

\item
If $n$ is odd, then
\begin{align*}
\binom{n}{k}\frac{n-2k}{2^k}
\sum_{i=0}^{\frac{n-1}{2}-k}(-\frac{1}{2})^i
\frac{\binom{n-2k-i}{i}}{n-2k-i}
\big(1^{\otimes(n-2k-2i)}\otimes\square^{\otimes(k+i)}\big)^\vee\cdot e^{(n)}_{\ell,\mathbb{H}}e^{(n)}_{m,\mathbb{O}}\,,
\end{align*}
$k\in\{n-\ell,\ldots,\frac{n-1}{2}\}$, $\ell\in\left\{\frac{n+1}{2},\ldots,m\right\}$,
form a complete set of primitive orthogonal idempotents in $\Sym^n\mathbb{H}\cdot e^{(n)}_{m,\mathbb{O}}$.

\item
If $n$ is even, then
\[
\left(\frac{1}{2}\right)^{\frac{n}{2}}\binom{n}{\frac{n}{2}}
\big(\square^{\otimes\frac{n}{2}}\big)^\vee\cdot e_{\frac{n}{2},\mathbb{H}}^{(n)}e_{m,\mathbb{O}}^{(n)}
\]
and
\[
\left(\frac{1}{2}\right)^{\frac{n}{2}}\binom{n}{\frac{n}{2}}
\big(\square^{\otimes\frac{n}{2}}\big)^\vee\cdot e_{\ell,\mathbb{H}}^{(n)}e_{m,\mathbb{O}}^{(n)}
\]
together with
\begin{align*}
\binom{n}{k}\frac{\frac{n}{2}-k}{2^k}
\sum_{i=0}^{\left\lfloor \frac{n}{4}-\frac{k}{2}\right\rfloor}
\left(-\frac{1}{4}\right)^i
\frac{\binom{\frac{n}{2}-k-i}{i}}{\frac{n}{2}-k-i}
&\big(\square^{\otimes(k+2i)}\otimes(1\otimes1-\square)^{\otimes(\frac{n}{2}-k-2i)}\big)^\vee\\
&\cdot\frac{1}{2}\big(1^{\otimes n}+\delta\cdot e_2^{\otimes n}\big)\cdot e^{(n)}_{\ell,\mathbb{H}}e^{(n)}_{m,\mathbb{O}}\,,
\end{align*}
$k\in\{n-\ell,\ldots,\frac{n}{2}-1\}$, $\delta\in\{-1,1\}$, $\ell\in\left\{\frac{n}{2}+1,\ldots,m\right\}$,
if $m>\frac{n}{2}$,
form a complete set of primitive orthogonal idempotents in $\Sym^n\mathbb{H}\cdot e^{(n)}_{m,\mathbb{O}}$.
\end{enumerate}
\end{theorem}

It follows that the size of a complete set of primitive orthogonal idempotents 
in \begin{center}$\bigoplus_{m=\frac{n}{2}}^n(\Sym^n\mathbb{H}\otimes_\mathbb{R}\mathbb{C})\cdot e^{(n)}_{m,\mathbb{O}}$\end{center} is
\[\begin{cases}
    \sum_{m=\frac{n}{2}}^n\frac{1}{4}(2m-n+2)^2=\frac{1}{24}(n+2)(n+3)(n+4),
    & \text{if $n$ is even}, \\
    \sum_{m=\frac{n+1}{2}}^n\frac{1}{4}(2m-n+1)(2m-n+3)=\frac{1}{24}(n+1)(n+3)(n+5),
    & \text{if $n$ is odd}.
\end{cases}\]
Also, the size of a complete set of primitive orthogonal idempotents in 
$\bigoplus_{m=\frac{n}{2}}^n\Sym^n\mathbb{H}\cdot e^{(n)}_{m,\mathbb{O}}$ is
\[\begin{cases}
    \sum_{m=\frac{n}{2}}^n\frac{1}{4}(2m-n+2)^2=\frac{1}{24}(n+2)(n+3)(n+4),
    & \text{if $n$ is even}, \\
    \sum_{m=\frac{n+1}{2}}^n\frac{1}{8}(2m-n+1)(2m-n+3)=\frac{1}{48}(n+1)(n+3)(n+5),
    & \text{if $n$ is odd}.
\end{cases}\]
By \eqref{eq:decomp}, applied to $\mathcal{C}=\mathbb{O}$, the center $Z(\Sym^n\mathbb{O})$ of $\Sym^n\mathbb{O}$
is $\bigoplus_{m=\left\lceil \frac{n}{2} \right\rceil}^n\mathbb{R}\cdot e^{(n)}_{m,\mathbb{O}}$,
since the $\mathbb{S}^{(n)}_m\mathbb{O}$'s are central algebras over $\mathbb{R}$.
Hence, $\Sym^n\mathbb{H}\cdot Z(\Sym^n\mathbb{O})
=\bigoplus_{m=\left\lceil \frac{n}{2} \right\rceil}^n\Sym^n\mathbb{H}\cdot e^{(n)}_{m,\mathbb{O}}$.
The following corollary of Theorem~\ref{TheoremC}
describes explicitly a complete set of primitive orthogonal idempotents in 
$\Sym^n\mathbb{H}\cdot Z(\Sym^n\mathbb{O})$
and in $\big(\Sym^n\mathbb{H}\cdot Z(\Sym^n\mathbb{O})\big)\otimes_\mathbb{R}\mathbb{C}$.

\begin{corollary}[{Corollary~\ref{Corollary4}}]
\label{CorollaryD}
\begin{enumerate}[label={\rm(\alph*)},leftmargin=*]
\item
The set
\[
\left\{
\binom{n}{k}\bigl(a^{\otimes k}\otimes(a^c)^{\otimes(n-k)}\bigr)^\vee 
e^{(n)}_{\ell,\mathbb{H}}e^{(n)}_{m,\mathbb{O}}
\;\middle|\;
\begin{aligned}
&k\in\{n-\ell,\ldots,\ell\},\\
&\ell\in\left\{\left\lceil \frac{n}{2} \right\rceil,\ldots,m\right\},\\
&m\in\left\{\left\lceil \frac{n}{2} \right\rceil,\ldots,n\right\}
\end{aligned}
\right\}
\]
is a complete set of primitive orthogonal idempotents in 
$\big(\Sym^n\mathbb{H}\cdot Z(\Sym^n\mathbb{O})\big)\otimes_\mathbb{R}\mathbb{C}$.

\item
If $n$ is odd, then
\begin{align*}
\binom{n}{k}\frac{n-2k}{2^k}
\sum_{i=0}^{\frac{n-1}{2}-k}\left(-\frac{1}{2}\right)^i
\frac{\binom{n-2k-i}{i}}{n-2k-i}
\big(1^{\otimes(n-2k-2i)}\otimes\square^{\otimes(k+i)}\big)^\vee\cdot 
e^{(n)}_{\ell,\mathbb{H}}e^{(n)}_{m,\mathbb{O}}\,,
\end{align*}
$k\in\{n-\ell,\ldots,\frac{n-1}{2}\}$, $\ell\in\left\{\frac{n+1}{2},\ldots,m\right\}$, $m\in\left\{\frac{n+1}{2},\ldots,n\right\}$,
form a complete set of primitive orthogonal idempotents in 
$\Sym^n\mathbb{H}\cdot Z(\Sym^n\mathbb{O})$.

\item
If $n$ is even, then
\[
\left(\frac{1}{2}\right)^{\frac{n}{2}}
\binom{n}{\frac{n}{2}}
\bigl(\square^{\otimes \frac{n}{2}}\bigr)^\vee
\cdot
e_{\frac{n}{2},\mathbb{H}}^{(n)}e_{m,\mathbb{O}}^{(n)},
\]
for $m\in\left\{\frac{n}{2},\ldots,n\right\}$,
and
\[
\left(\frac{1}{2}\right)^{\frac{n}{2}}
\binom{n}{\frac{n}{2}}
\bigl(\square^{\otimes \frac{n}{2}}\bigr)^\vee
\cdot
e_{\ell,\mathbb{H}}^{(n)}e_{m,\mathbb{O}}^{(n)}
\]
together with
\begin{align*}
\binom{n}{k}\frac{\frac{n}{2}-k}{2^k}
\sum_{i=0}^{\left\lfloor \frac{n}{4}-\frac{k}{2} \right\rfloor}
\left(-\frac{1}{4}\right)^i
\frac{\binom{\frac{n}{2}-k-i}{i}}{\frac{n}{2}-k-i}
&\big(\square^{\otimes(k+2i)}\otimes(1\otimes1-\square)^{\otimes(\frac{n}{2}-k-2i)}\big)^\vee\\
&\cdot\frac{1}{2}\big(1^{\otimes n}+\delta\cdot e_2^{\otimes n}\big)\cdot e^{(n)}_{\ell,\mathbb{H}}e^{(n)}_{m,\mathbb{O}}\,,
\end{align*}
$k\in\{n-\ell,\ldots,\frac{n}{2}-1\}$, $\delta\in\{-1,1\}$, $\ell\in\left\{\frac{n}{2}+1,\ldots,m\right\}$,
$m\in\left\{\frac{n}{2}+1,\ldots,n\right\}$,
form a complete set of primitive orthogonal idempotents in 
$\Sym^n\mathbb{H}\cdot Z(\Sym^n\mathbb{O})$.
\end{enumerate}
\end{corollary}

\section[Primitive orthogonal idempotents]{Complete set of primitive orthogonal idempotents}
\label{I}

Let $R$ be an associative ring, not necessarily commutative, with a unit $1$.
An element $x\in R$ is an {\bf idempotent} if $x^2=x$.
A set of idempotents $\{e_i\}$ is said to be {\bf orthogonal} if $e_ie_j=0$ for all $i\ne j$.
Clearly, the sum of two orthogonal idempotents is also an idempotent.
Also, if $e$ is any idempotent, then $e$ and $1-e$ are orthogonal idempotents.
The key result is the following.

\begin{lemma}[see {\cite[p.\ 119, 1st par.]{CuR81}}]
\label{KeyIdem}
There is a $1\!\!-\!\!1$ correspondence between a decomposition $R=I_1\oplus\ldots\oplus I_n$
as a direct sum of left ideals and
orthogonal idempotents $e_1,\ldots,e_n$ such that $e_1+\ldots+e_n=1$.
\end{lemma}

Note that under the above correspondence, $I_i=\{0\}$ if and only if $e_i=0$.
Also, any idempotent $e\in R$ gives orthogonal $\{e,1-e\}$, therefore $R=Re\oplus R(1-e)$.

Now suppose that $R$ is an artinian ring (and hence noetherian by the Hopkins-Levitzki theorem).
The Krull-Schmidt theorem \cite[p.\ 128, Thm.\ (6.12)]{CuR81}
says that $R$ is a direct sum of finitely many indecomposable projective modules $I_i$.
Such modules correspond to primitive idempotents.

\begin{definition}
\label{primitive}
A {\sl nonzero} idempotent $e$ is said to be {\bf primitive} 
if it cannot be written as a sum of nonzero orthogonal idempotents $e=f_1+f_2$.
\end{definition}

\begin{lemma}[see {\cite[p.\ 119, 2nd par.]{CuR81}}]
\label{indecomposable}
If $e$ is an idempotent, then $Re$ is indecomposable if and only if $e$ is primitive.
\end{lemma}

Note that in writing $R$ as a direct sum of indecomposable projective modules, the terms are unique
up to isomorphism and permutation, but this does {\sl not} mean that the corresponding idempotents are unique.
However, their number is the same.

\begin{definition}
\label{complete}
We say that a set $\{e_1,\ldots,e_n\}$ of orthogonal idempotents in $R$ is {\bf complete}
if $e_1+\ldots+e_n=1$.
\end{definition}

By Lemma~\ref{KeyIdem} and Lemma~\ref{indecomposable} we get the following corollary.

\begin{corollary}
\label{cpoi}
There is a $1\!\!-\!\!1$ correspondence between a decomposition $R=I_1\oplus\ldots\oplus I_n$
as a direct sum of indecomposable left ideals and a complete set of primitive orthogonal idempotents.
In particular, if $\{e_1,\ldots,e_r\}$ and $\{f_1,\ldots,f_s\}$ are
two complete sets of primitive orthogonal idempotents in $R$, then $r=s$.
\end{corollary}

\begin{example}
\label{matrix}
In the ring $M_n(D)$ of $n\times n$ matrices over a division ring $D$,
the set of diagonal matrices $\{E_{11},\ldots,E_{nn}\}$ 
(where $E_{ii}$ has $1$ at $(i,i)$ and $0$ elsewhere)
is a complete set of primitive orthogonal idempotents.

Indeed, this follows from Corollary~\ref{cpoi} since
$M_n(D)=M_n(D)E_{11}\oplus\ldots\oplus M_n(D)E_{nn}$
is a decomposition of $M_n(D)$ as a direct sum of indecomposable left ideals.
\end{example}

Next, we look at central idempotents.

\begin{definition}
\label{central}
An idempotent $e$ of $R$ is said to be {\bf central} if it lies in the center of $R$.
We say that $e$ is {\bf centrally primitive} in $R$ if $e\ne0$ and 
$e$ cannot be written as a sum of two nonzero orthogonal central idempotents in $R$.
\end{definition}

\begin{lemma}[\!{\cite[pp.\ 326--328, Prop.\ (22.1) and Prop.\ (22.2)]{Lam01}}]
\label{centralLemma}
There is a bijection between:
\begin{enumerate}[label={\rm(\alph*)},leftmargin=*]
\item
an isomorphism $R\cong R_1\times\cdots\times R_n$ as a product of indecomposable rings;

\item
a decomposition $R=I_1\oplus\ldots\oplus I_n$ as a direct sum of indecomposable (two-sided) ideals; and

\item
an expression $1=e_1+\ldots+e_n$ as a sum of orthogonal centrally primitive idempotents.
\end{enumerate}
The correspondence ${\rm(c)}\Rightarrow{\rm(b)}$ is given by $I_i=Re_i$, $i=1,\ldots,n$.
\end{lemma}

\begin{corollary}
\label{completeCorollary}
Suppose that $\{e_1,\ldots,e_n\}$ is  a set of orthogonal central idempotents in $R$ 
with $1=e_1+\ldots+e_n$.
If $\{e_{ij}\}_{j=1}^{m_i}$ is a complete set of primitive orthogonal idempotents in the ring $Re_i$ (with unit $e_i$),
$i=1,\ldots,n$, then
\begin{equation}
\label{completeEquation}
\big\{e_{11},\ldots,e_{1,m_1},\ldots,e_{n1},\ldots,e_{n,m_n}\big\}
\end{equation}
is a complete set of primitive orthogonal idempotents in $R$.
\end{corollary}

\begin{proof}
According to the correspondence given in Lemma~\ref{centralLemma},
$R=Re_1\oplus\ldots\oplus Re_n$ is a direct sum of (two-sided) ideals.
Hence the set in \eqref{completeEquation} consists of orthogonal idempotents.
Since $e_i=e_{i1}+\ldots+e_{i,m_i}$, $i=1,\ldots,n$, we get
$1=e_{11}+\ldots+e_{1,m_1}+\ldots+e_{n1}+\ldots+e_{n,m_n}$,
so this set is complete.

Moreover, each $e_{ij}$ is primitive in $R$.

Indeed, as $e_i=e_{i1}+\ldots+e_{i,m_i}$ is a sum of orthogonal idempotents, $e_ie_{ij}=e_{ij}$.
Therefore, since $e_{ij}$ is primitive in $Re_i$,
we get by Lemma~\ref{indecomposable} that $Re_{ij}=Re_ie_{ij}$ is indecomposable in the ring $Re_i$
as a direct sum of two nonzero left ideals.
If $I$ is a left ideal of $R$ which is contained in $Re_{ij}$, then it is also
a left ideal of $Re_i$ which is contained in $Re_ie_{ij}$.
Hence, $Re_{ij}$ is also indecomposable in the ring $R$.
Thus, by Lemma~\ref{indecomposable} again, $e_{ij}$ is primitive in $R$.
\end{proof}


\section[Symmetric tensor powers]{Symmetric tensor powers of quaternions and octonions}
\label{II}

Let $\mathbb{O}$ be the Cayley octonion algebra of dimension $8$ over $\mathbb{R}$ 
with a basis $e_0=1$, $e_1$, $e_2$, $e_3,e_4,e_5,e_6,e_7$
that satisfies the following multiplication table:

\begin{equation}
\label{MultTable}
\begin{tabular}{ c| c | c | c | c | c | c | c | c |}
$\cdot$ & $e_0$ & $e_1$ & $e_2$ & $e_3$ & $e_4$ & $e_5$ & $e_6$ & $e_7$ \\
\hline
$e_0$ & $e_0$ & $e_1$ & $e_2$ & $e_3$ & $e_4$ & $e_5$ & $e_6$ & $e_7$ \\
\hline
$e_1$ & $e_1$ & $-e_0$ & $e_3$ & $-e_2$ & $e_5$ & $-e_4$ & $-e_7$ & $e_6$ \\
\hline
$e_2$ & $e_2$ & $-e_3$ & $-e_0$ & $e_1$ & $e_6$ & $e_7$ & $-e_4$ & $-e_5$ \\
\hline
$e_3$ & $e_3$ & $e_2$ & $-e_1$ & $-e_0$ & $e_7$ & $-e_6$ & $e_5$ & $-e_4$ \\
\hline
$e_4$ & $e_4$ & $-e_5$ & $-e_6$ & $-e_7$ & $-e_0$ & $e_1$ & $e_2$ & $e_3$ \\
\hline
$e_5$ & $e_5$ & $e_4$ & $-e_7$ & $e_6$ & $-e_1$ & $-e_0$ & $-e_3$ & $e_2$ \\
\hline
$e_6$ & $e_6$ & $e_7$ & $e_4$ & $-e_5$ & $-e_2$ & $e_3$ & $-e_0$ & $-e_1$ \\
\hline
$e_7$ & $e_7$ & $-e_6$ & $e_5$ & $e_4$ & $-e_3$ & $-e_2$ & $e_1$ & $-e_0$ \\
\hline
\end{tabular}
\end{equation}

\bigskip\noindent
We let $\mathbb{H}$ be the Hamilton quaternion algebra of dimension $4$ over $\mathbb{R}$,
considered as a subalgebra of $\mathbb{O}$,
with the basis $e_0=1,e_1,e_2,e_3$.
We also consider the algebra of complex quaternions $\mathbb{H}_\mathbb{C}:=\mathbb{H}\otimes_\mathbb{R}\mathbb{C}$ 
and the algebra of complex octonions $\mathbb{O}_\mathbb{C}:=\mathbb{O}\otimes_\mathbb{R}\mathbb{C}$ 
over the field $\mathbb{C}=\mathbb{R}[\sqrt{-1}]$ of complex numbers,
where $\sqrt{-1}$ is a root of $X^2+1$ in $\mathbb{C}$.
We write any $x\in\mathbb{O}_\mathbb{C}$ as a sum $x=\sum_{i=0}^7\alpha_ie_i$, with $\alpha_i\in\mathbb{C}$,
and define a nondegenerate quadratic form $N\colon\mathbb{O}_\mathbb{C}\rightarrow\mathbb{C}$ by
$N(x)=\sum_{i=0}^7\alpha_i^2$.
It is multiplicative, meaning that $N(xy)=N(x)N(y)$ for all $x,y\in\mathbb{O}_\mathbb{C}$.
Its restrictions to $\mathbb{O}$, $\mathbb{H}$, $\mathbb{R}[e_1]$, $\mathbb{H}_\mathbb{C}$, and $\mathbb{R}[e_1]_\mathbb{C}:=\mathbb{C}[e_1]$ 
are also multiplicative nondegenerate quadratic forms.
Thus, $\mathbb{R}[e_1]$, $\mathbb{H}$, and $\mathbb{O}$ (resp.\ $\mathbb{C}[e_1]$, $\mathbb{H}_\mathbb{C}$, and $\mathbb{O}_\mathbb{C}$) 
are {\bf composition} algebras over $\mathbb{R}$ (resp.\ $\mathbb{C}$).

We denote by $\mathcal{C}$ either the algebra $\mathbb{R}[e_1]$, which is isomorphic to $\mathbb{C}$, or
the quaternion algebra $\mathbb{H}=\mathbb{R}[e_1,e_2,e_3]$ or the octonion algebra 
$\mathbb{O}=\mathbb{R}[e_1,e_2,e_3,e_4,e_5,e_6,e_7]$ over $\mathbb{R}$.
We also denote by $F$ either the field $\mathbb{R}$ of real numbers or the field $\mathbb{C}$ of complex numbers.
Then, with $\mathcal{C}_\mathbb{R}:=\mathcal{C}$, the algebra $\mathcal{C}_F$ is a composition algebra over $F$,
which is one of the algebras $\mathbb{R}[e_1]$, $\mathbb{H}$, $\mathbb{O}$, $\mathbb{C}[e_1]$, $\mathbb{H}_\mathbb{C}$, $\mathbb{O}_\mathbb{C}$.
Let $d=\dim_\mathbb{R}\mathcal{C}=\dim_\mathbb{C}\mathcal{C}_\mathbb{C}$, so $d=2$ if $\mathcal{C}=\mathbb{R}[e_1]$ or
$d=4$ if $\mathcal{C}=\mathbb{H}$ or $d=8$ if $\mathcal{C}=\mathbb{O}$.
Consider the element
\begin{equation}
\label{triangle}
\triangle_\mathcal{C}:=\frac{1}{d}\sum_{i=0}^{d-1}e_i\otimes e_i
\end{equation}
of the ring of symmetric tensors $\Sym^2\mathcal{C}$ in $\mathcal{C}\otimes_\mathbb{R}\mathcal{C}$.
It satisfies the following result.

\begin{lemma}
\label{triangleLemma}
$\triangle_\mathcal{C}(x\otimes x)=N(x)\triangle_\mathcal{C}=(x\otimes x)\triangle_\mathcal{C}$ for all $x\in\mathcal{C}_F$.
In particular, $\triangle_\mathcal{C}$ lies in the {\bf center} of $\Sym^2\mathcal{C}_F$, 
i.e. it commutes and associates with all elements of $\Sym^2\mathcal{C}_F$,
and $\Sym^2\mathcal{C}_F\cdot\triangle_\mathcal{C}=F\cdot\triangle_\mathcal{C}$.
\end{lemma}

\begin{proof}
This is \cite[Lemma 2.2]{Raz24} for the cases $\mathcal{C}=\mathbb{H}$ or $\mathcal{C}=\mathbb{O}$.
For $\mathcal{C}=\mathbb{R}[e_1]$, this follows from
\begin{equation*}
\frac{1}{2}(1\otimes1+e_1\otimes e_1)\big((\alpha+\beta e_1)\otimes(\alpha+\beta e_1)\big)
=\frac{1}{2}(\alpha^2+\beta^2)(1\otimes1+e_1\otimes e_1),
\end{equation*}
with $\alpha,\beta\in\mathbb{C}$, by a direct calculation.
\end{proof}

In particular, by Lemma~\ref{triangleLemma}, $\triangle_\mathcal{C}(e_i\otimes e_i)=\triangle_\mathcal{C}$, $i=0,\ldots,d-1$.
Hence, by \eqref{triangle},
\begin{equation}
\label{triangleSq}
\triangle_\mathcal{C}^2=\triangle_\mathcal{C},
\end{equation}
\begin{equation}
\label{triangleReH}
\triangle_\mathbb{H}\triangle_{\mathbb{R}[e_1]}=\triangle_\mathbb{H}=\triangle_{\mathbb{R}[e_1]}\triangle_\mathbb{H},
\end{equation}
and
\begin{equation}
\label{triangleHO}
\triangle_\mathbb{O}\triangle_\mathbb{H}=\triangle_\mathbb{O}=\triangle_\mathbb{H}\triangle_\mathbb{O}.
\end{equation}

We denote
\begin{equation*}
\Imag_1\mathcal{C}_F:=\{e\in\mathcal{C}_F\mid e^2=-1\}
=\Big\{\sum_{i=1}^{d-1}\alpha_ie_i\mid\sum_{i=1}^{d-1}\alpha_i^2=1,\alpha_i\in F\Big\}.
\end{equation*}
For $e\in\Imag_1\mathcal{C}_F$ and $x\in\mathcal{C}_F$, we denote
\begin{equation*}
x_e:=\frac{1}{2}(x-exe).
\end{equation*}

\begin{lemma}
\label{epart}
Let $e=\sum_{i=1}^{d-1}\alpha_ie_i\in\Imag_1\mathcal{C}_F$. Then, $(e_0)_e=1$ and $(e_i)_e=\alpha_ie$, for all $i=1,\ldots,d-1$.
\end{lemma}

\begin{proof}
This is \cite[Lemma 2.3]{Raz24} for the cases $\mathcal{C}=\mathbb{H}$ or $\mathcal{C}=\mathbb{O}$. For $\mathcal{C}=\mathbb{R}[e_1]$, we have 
$\Imag_1\mathcal{C}_F=\{-e_1,e_1\}$, so in this case $(e_0)_{e_1}=(e_0)_{-e_1}=1$ 
and $(e_1)_{e_1}=(e_1)_{-e_1}=e_1$, by definition.
\end{proof}

Fix a positive integer $n$.
Let $\mathcal{C}_F^{\otimes n}$ be the tensor product over $F$ of $n$ copies of $\mathcal{C}_F$. For $e\in\Imag_1\mathcal{C}_F$,
we extend the $F$-linear map $\mathcal{C}_F\rightarrow F[e]$ defined by $x\mapsto x_e$ to 
an $F$-linear map $\mathcal{C}_F^{\otimes n}\rightarrow F[e]$ defined by 
\begin{equation}
\label{epartn}
x_1\otimes\cdots\otimes x_n\mapsto(x_1\otimes\cdots\otimes x_n)_e:=(x_1)_e\cdots(x_n)_e.
\end{equation}

Define the action of the symmetric group 
$\mathfrak{S}_n$ on $\mathcal{C}_F^{\otimes n}$ by 
\[\sigma(x_1\otimes\cdots\otimes x_n)=x_{\sigma^{-1}(1)}\otimes\cdots\otimes x_{\sigma^{-1}(n)}\]
for $\sigma\in \mathfrak{S}_n$ and $x_1,\dots,x_n\in\mathcal{C}_F$.
Let
\begin{equation*}
x^\vee:=\frac{1}{n!}\sum_{\sigma\in\mathfrak{S}_n}\sigma(x)
\end{equation*}
for $x\in\mathcal{C}_F^{\otimes n}$.
Then $\{x^\vee\mid x\in\mathcal{C}_F^{\otimes n}\}$ is the
fixed ring
\begin{equation*}
\Sym^n\mathcal{C}_F=\{x\in\mathcal{C}_F^{\otimes n}\mid\sigma(x)=x\text{ for all }\sigma\in\mathfrak{S}_n\}
\end{equation*}
of $\mathcal{C}_F^{\otimes n}$ under the action of $\mathfrak{S}_n$. 
We identify $\Sym^n\mathcal{C}_\mathbb{C}$ with $\Sym^n\mathcal{C}\otimes_\mathbb{R}\mathbb{C}$.
Note that
\begin{subequations}\label{eq_sym_1}
\begin{equation}\label{eq_sym_1a}
(\sigma(x))^\vee=x^\vee\hbox{\rm\ for\ }x\in\mathcal{C}_F^{\otimes n}\text{ and }\sigma\in\mathfrak{S}_n\,,
\end{equation}
\begin{equation}\label{eq_sym_1b}
(x\otimes y)^\vee=(x\otimes y^\vee)^\vee\hbox{\rm\ for\ }x\in\mathcal{C}_F^{\otimes k},
y\in\mathcal{C}_F^{\otimes(n-k)},\text{ and}
\end{equation}
\begin{equation}\label{eq_sym_1c}
(xr)^\vee=x^\vee r\text{ and }(rx)^\vee=rx^\vee\hbox{\rm\ for\ }x\in\mathcal{C}_F^{\otimes n},r\in\Sym^n\mathcal{C}_F.
\end{equation}
\end{subequations}
Clearly, $(x^\vee)_e=x_e$ for each $x\in\mathcal{C}_F^{\otimes n}$.

In the rest of this section we summarize the results of \cite{Raz21,  Raz24, Raz25}
that we need on the decomposition of $\Sym^n\mathcal{C}_F$ into its simple components
and draw from them more results.

\begin{lemma}[see {\cite[Lemma 3.1]{Raz21}}]
\label{span} 
The set $\{x^{\otimes n}\mid x\in\mathcal{C}_F\}$ spans $\Sym^n\mathcal{C}_F$.
\end{lemma}

We write $\left\lceil \frac{n}{2} \right\rceil$ (resp. $\left\lfloor \frac{n}{2} \right\rfloor$) for the integer rounding up (resp. down) $\frac{n}{2}$.
\begin{lemma}
\label{zeroCond}
Let $\ell$ be an integer between $0$ and $\left\lfloor \frac{n}{2} \right\rfloor$. Then
$(s\otimes\triangle_\mathcal{C}^{\otimes\ell})^\vee=0$ if and only if $s=0$ for each $s\in\Sym^{n-2\ell}\mathcal{C}_F$.
\end{lemma}

\begin{proof}
The case $\ell=1$ is \cite[Lemma 3.3]{Raz24} for $\mathcal{C}=\mathbb{H}$ or $\mathcal{C}=\mathbb{O}$
and it holds also for $\mathcal{C}=\mathbb{R}[e_1]$ by a similar proof.
Applying it $\ell$ times, using \eqref{eq_sym_1b}, proves the lemma.
\end{proof}

For the rest of this work, $\mathcal{C}=\mathbb{H}$ or $\mathcal{C}=\mathbb{O}$.
We suppose that $n\ge2$. Denote
\begin{equation}
\label{trianglen}
\triangle_\mathcal{C}^{(n)}=(\triangle_\mathcal{C}\otimes1^{\otimes(n-2)})^\vee
\end{equation}
in $\Sym^n\mathcal{C}$.

\begin{proposition}[see {\cite[Prop.\ 5.3]{Raz24}}]
\label{LocalGlobal}
Let $r\in\Sym^n\mathcal{C}$. Then, $r\in\Sym^n\mathcal{C}\cdot\triangle_\mathcal{C}^{(n)}$ 
if and only if $r_e=0$ for each $e\in\Imag_1\mathcal{C}$.
\end{proposition}

For the next result we need the following combinatorial lemma of Noga Alon:
\begin{lemma}[see {\cite[Lemma 2.1]{Alo99}}]
\label{Noga}
Let $P=P(X_1,\dots,X_k)$ be a polynomial in $k$ variables over an arbitrary field $K$.
Suppose that the degree of $P$ as a polynomial in $X_i$ is at most $m_i$ for $1\le i\le k$, 
and let $S_i\subset K$ be a set of at least $m_i+1$ distinct elements of $K$.
If $P(x_1,\dots,x_k)=0$ for all $k$-tuples $(x_1,\dots,x_k)\in S_1\times\cdots\times S_k$, then $P\equiv0$.
\end{lemma}

\begin{lemma}
\label{zeroIntersection}
\begin{enumerate}[label={\rm(\alph*)},leftmargin=*]
\item
$\Sym^n\mathbb{R}[e_1]\cap\Sym^n\mathbb{H}\cdot\triangle_\mathbb{H}^{(n)}=\{0\}$.

\item
$\Sym^n\mathbb{H}\cap\Sym^n\mathbb{O}\cdot\triangle_\mathbb{O}^{(n)}=\{0\}$.
\end{enumerate}
\end{lemma}

\begin{proof}
The set $\big\{(e_0^{\otimes k_0}\otimes e_1^{\otimes k_1})^\vee\mid k_0+k_1=n\big\}$
forms a basis of $\Sym^n\mathbb{R}[e_1]$ over $\mathbb{R}$.
Let $r\in\Sym^n\mathbb{R}[e_1]$ and write
\begin{equation*}
r=\sum_{k_0+k_1=n}\beta_{(k_0,k_1)}\cdot(e_0^{\otimes k_0}\otimes e_1^{\otimes k_1})^\vee,
\end{equation*}
with $\beta_{(k_0,k_1)}\in\mathbb{R}$.
By Proposition~\ref{LocalGlobal}, applied to $\mathcal{C}=\mathbb{H}$,
$r\in\Sym^n\mathbb{H}\cdot\triangle_\mathbb{H}^{(n)}$ if and only if 
$r_e=0$ for each $e\in\Imag_1\mathbb{H}$.

Let $e=\sum_{i=1}^3\alpha_ie_i\in\Imag_1\mathbb{H}$, with $\alpha_i\in\mathbb{R}$ that satisfy $\sum_{i=1}^3\alpha_i^2=1$.
By \eqref{epartn} and Lemma~\ref{epart}, applied to $\mathcal{C}=\mathbb{H}$,
\begin{equation}
\label{re}
r_e=\sum_{k_0+k_1=n}\beta_{(k_0,k_1)}\alpha_1^{k_1}\cdot e^{k_1}.
\end{equation}
Since $e^2=-1$, the map $e\mapsto\sqrt{-1}$ defines a field isomorphism $\mathbb{R}[e]\cong\mathbb{C}=\mathbb{R}[\sqrt{-1}]$.
By \eqref{re}, $r_e=0$ if and only if
\begin{equation}
\label{re0}
\sum_{k_0+k_1=n}\beta_{(k_0,k_1)}(\sqrt{-1}\alpha_1)^{k_1}=0.
\end{equation}

Consider the polynomial $Q(Y) =\sum_{k_0+k_1=n}\beta_{(k_0,k_1)}\big(\sqrt{-1}Y\big)^{k_1}$ in $\mathbb{C}[Y]$.
By \eqref{re0}, $r_e=0$ for each $e\in\Imag_1\mathbb{H}$ if and only if
$Q(\alpha_1)=0$ for each $(\alpha_1,\alpha_2,\alpha_3)\in\mathbb{R}^3$ with $\alpha_1^2+\alpha_2^2+\alpha_3^2=1$.
Let $S=\{\alpha\in\mathbb{R}\mid\alpha^2\le1\}$. Then, for each $\alpha_1\in S$, 
we can find $\alpha_2,\alpha_3\in\mathbb{R}$ with $\alpha_1^2+\alpha_2^2+\alpha_3^2=1$.
Thus, $r\in\Sym^n\mathbb{H}\cdot\triangle_\mathbb{H}^{(n)}$ if and only if
$Q(\alpha_1)=0$ for each $\alpha_1\in S$.
In this case, it follows from Lemma~\ref{Noga} that $Q\equiv0$, 
so $\beta_{(k_0,k_1)}=0$ for all $(k_0,k_1)$ with $k_0+k_1=n$, which implies $r=0$.
This concludes the proof of (a).

The proof of (b) is done in a similar way.
The set 
\begin{equation*}
\big\{(e_0^{\otimes k_0}\otimes e_1^{\otimes k_1}\otimes e_2^{\otimes k_2}\otimes e_3^{\otimes k_3})^\vee
\mid k_0+k_1+k_2+k_3=n\big\}
\end{equation*}
forms a basis of $\Sym^n\mathbb{H}$ over $\mathbb{R}$.
Let $r\in\Sym^n\mathbb{H}$ and write
\begin{equation*}
r=\sum_{k_0+k_1+k_2+k_3=n}\beta_{(k_0,k_1,k_2,k_3)}\cdot
(e_0^{\otimes k_0}\otimes e_1^{\otimes k_1}\otimes e_2^{\otimes k_2}\otimes e_3^{\otimes k_3})^\vee,
\end{equation*}
with $\beta_{(k_0,k_1,k_2,k_3)}\in\mathbb{R}$.
By Proposition~\ref{LocalGlobal}, applied to $\mathcal{C}=\mathbb{O}$,
$r\in\Sym^n\mathbb{O}\cdot\triangle_\mathbb{O}^{(n)}$ if and only if 
$r_e=0$ for each $e\in\Imag_1\mathbb{O}$.

Let $e=\sum_{i=1}^7\alpha_ie_i\in\Imag_1\mathbb{O}$, with $\alpha_i\in\mathbb{R}$ that satisfy $\sum_{i=1}^7\alpha_i^2=1$.
By \eqref{epartn} and Lemma~\ref{epart}, applied to $\mathcal{C}=\mathbb{O}$,
\begin{equation}
\label{reH}
r_e=\sum_{k_0+k_1+k_2+k_3=n}\beta_{(k_0,k_1,k_2,k_3)}\alpha_1^{k_1}\alpha_2^{k_2}\alpha_3^{k_3}
\cdot e^{k_1+k_2+k_3}.
\end{equation}
As in part (a), we get by \eqref{reH} that $r_e=0$ if and only if
\begin{equation}
\label{reH0}
\sum_{k_0+k_1+k_2+k_3=n}\beta_{(k_0,k_1,k_2,k_3)}
(\sqrt{-1}\alpha_1)^{k_1}(\sqrt{-1}\alpha_2)^{k_2}(\sqrt{-1}\alpha_3)^{k_3}=0.
\end{equation}

Consider the polynomial 
\begin{equation*}
Q(Y_1,Y_2,Y_3) =\sum_{k_0+k_1+k_2+k_3=n}
\beta_{(k_0,k_1,k_2,k_3)}\big(\sqrt{-1}Y_1\big)^{k_1}\big(\sqrt{-1}Y_2\big)^{k_2}\big(\sqrt{-1}Y_3\big)^{k_3}
\end{equation*}
in $\mathbb{C}[Y_1,Y_2,Y_3]$.
By \eqref{reH0}, $r_e=0$ for each $e\in\Imag_1\mathbb{O}$ if and only if
$Q(\alpha_1,\alpha_2,\alpha_3)=0$ for each 
$(\alpha_1,\alpha_2,\alpha_3,\alpha_4,\alpha_5,\alpha_6,\alpha_7)\in\mathbb{R}^7$ 
with $\sum_{i=1}^7\alpha_i^2=1$.
Let $S=\{\alpha\in\mathbb{R}\mid\alpha^2\le\frac{1}{3}\}$.
Then, for each $(\alpha_1,\alpha_2,\alpha_3)\in S^3$, 
we can find $\alpha_4,\alpha_5,\alpha_6,\alpha_7\in\mathbb{R}$ with $\sum_{i=1}^7\alpha_i^2=1$.
Thus, $r\in\Sym^n\mathbb{O}\cdot\triangle_\mathbb{O}^{(n)}$ only if
$Q(\alpha_1,\alpha_2,\alpha_3)=0$ for each $(\alpha_1,\alpha_2,\alpha_3)\in S^3$.
In this case, it follows from Lemma~\ref{Noga} that $Q\equiv0$, 
so $\beta_{(k_0,k_1,k_2,k_3)}=0$ for all $(k_0,k_1,k_2,k_3)$ with $k_0+k_1+k_2+k_3=n$, which implies $r=0$.
This concludes the proof of (b).
\end{proof}

\begin{lemma}
\label{center}
$\triangle_\mathcal{C}\otimes1^{\otimes(n-2)}$ commutes and associates with all elements of $\Sym^n\mathcal{C}_F$.
In particular, $\triangle_\mathcal{C}^{(n)}$ lies in the center of $\Sym^n\mathcal{C}_F$.
Hence, $F[\triangle_\mathcal{C}^{(n)}]$ is a commutative and associative $F$-algebra.
\end{lemma}

\begin{proof}
We follow the proof of \cite[Lemma 4.1]{Raz21}.
By Lemma~\ref{triangleLemma}, for all $x,y\in\mathcal{C}_F$,
\begin{equation*}
(\triangle_\mathcal{C}\otimes1^{\otimes(n-2)})x^{\otimes n}
=N(x)(\triangle_\mathcal{C}\otimes x^{\otimes(n-2)})=x^{\otimes n}(\triangle_\mathcal{C}\otimes1^{\otimes(n-2)})
\end{equation*}
and
\begin{align*}
(\triangle_\mathcal{C}\otimes1^{\otimes(n-2)} \, x^{\otimes n})y^{\otimes n}
&=N(x)(\triangle_\mathcal{C}\otimes x^{\otimes(n-2)}) y^{\otimes n}\\
&=N(x)N(y)(\triangle_\mathcal{C}\otimes x^{\otimes(n-2)}y^{\otimes(n-2)})
= N(xy)(\triangle_\mathcal{C}\otimes(xy)^{\otimes(n-2)})\\
&=(\triangle_\mathcal{C}\otimes1^{\otimes(n-2)})(xy)^{\otimes n}
=(\triangle_\mathcal{C}\otimes1^{\otimes(n-2)})(x^{\otimes n}y^{\otimes n})\,.
\end{align*}
Similarly, we obtain the following identities:
\begin{align*}
(x^{\otimes n}(\triangle_\mathcal{C}\otimes 1^{\otimes(n-2)}))y^{\otimes n}
&= x^{\otimes n}((\triangle_\mathcal{C}\otimes 1^{\otimes(n-2)})y^{\otimes n}), \\
(x^{\otimes n}y^{\otimes n})(\triangle_\mathcal{C}\otimes 1^{\otimes(n-2)})
&= x^{\otimes n}(y^{\otimes n}(\triangle_\mathcal{C}\otimes 1^{\otimes(n-2)})).
\end{align*}

It follows from Lemma~\ref{span} that $\triangle_\mathcal{C}\otimes1^{\otimes(n-2)}$ 
commutes and associates with all elements of $\Sym^n\mathcal{C}_F$.
By \eqref{trianglen} and \eqref{eq_sym_1c}, $\triangle_\mathcal{C}^{(n)}$ also
commutes and associates with all elements of $\Sym^n\mathcal{C}_F$.
\end{proof}

For each integer $m$ between $\left\lceil \frac{n}{2} \right\rceil$ and $n$ we denote
\begin{equation}
\label{beta}
\beta_{m,\mathcal{C}}^{(n)}=\frac{2(n-m)(2m+d-2)}{dn(n-1)}.
\end{equation}
Then, $\beta_{m_1,\mathcal{C}}^{(n)}\ne\beta_{m_2,\mathcal{C}}^{(n)}$ for integers $m_1\ne m_2$ 
between $\left\lceil \frac{n}{2} \right\rceil$ and $n$ \cite[Sec.\ 3]{Raz24}.
The $\beta_{m,\mathcal{C}}^{(n)}$'s satisfy the following identity:
\begin{equation}
\label{betaIdentity}
n(n-1)\big(\beta_{k,\mathcal{C}}^{(n)}-\beta_{\ell,\mathcal{C}}^{(n)}\big)=
(n-2)(n-3)\big(\beta_{k-1,\mathcal{C}}^{(n-2)}-\beta_{\ell-1,\mathcal{C}}^{(n-2)}\big)
\end{equation}
for all integers $k,\ell$ between $\left\lceil \frac{n}{2} \right\rceil$ and $n$.
Indeed, direct calculation gives that both the left and right hand sides of \eqref{betaIdentity} are equal to
\begin{equation*}
\frac{2}{d}(k-\ell)\big(2n-(2k+2\ell+d-2)\big).
\end{equation*}

For each $m$ between $\left\lceil \frac{n}{2} \right\rceil$ and $n$, we denote
\begin{equation}
\label{emn}
e_{m,\mathcal{C}}^{(n)}=\prod_{\substack{\left\lceil \frac{n}{2} \right\rceil\le m'\le n\\ m'\ne m}}
\frac{\triangle_\mathcal{C}^{(n)}-\beta_{m',\mathcal{C}}^{(n)}\cdot1^{\otimes n}}
{\beta_{m,\mathcal{C}}^{(n)}-\beta_{m',\mathcal{C}}^{(n)}}.
\end{equation}
Note that by Lemma~\ref{center}, $e_{m,\mathcal{C}}^{(n)}$ is well-defined.
We also denote $e_{0,\mathcal{C}}^{(0)}=1_\mathbb{R}$ and $e_{1,\mathcal{C}}^{(1)}=1_\mathcal{C}$.
In particular, since $\beta_{1,\mathcal{C}}^{(2)}=1$ and $\beta_{2,\mathcal{C}}^{(2)}=0$,

\begin{subequations}\label{e2}
\begin{equation}
\label{e21}
e_{1,\mathcal{C}}^{(2)}=\frac{\triangle_\mathcal{C}-\beta_{2,\mathcal{C}}^{(2)}\cdot 1\otimes1}
{\beta_{1,\mathcal{C}}^{(2)}-\beta_{2,\mathcal{C}}^{(2)}}=\triangle_\mathcal{C},
\end{equation}
\begin{equation}
\label{e22}
e_{2,\mathcal{C}}^{(2)}=\frac{\triangle_\mathcal{C}-\beta_{1,\mathcal{C}}^{(2)}\cdot 1\otimes1}
{\beta_{2,\mathcal{C}}^{(2)}-\beta_{1,\mathcal{C}}^{(2)}}=1\otimes1-\triangle_\mathcal{C}.
\end{equation}
\end{subequations}
Similarly, since $\beta_{2,\mathcal{C}}^{(3)}=\frac{d+2}{3d}$ and $\beta_{3,\mathcal{C}}^{(3)}=0$,
\begin{subequations}\label{e3}
\begin{equation}
\label{e32}
e_{2,\mathcal{C}}^{(3)}=\frac{\triangle_\mathcal{C}^{(3)}-\beta_{3,\mathcal{C}}^{(3)}\cdot 1^{\otimes 3}}
{\beta_{2,\mathcal{C}}^{(3)}-\beta_{3,\mathcal{C}}^{(3)}}
=\frac{3d}{d+2}(\triangle_\mathcal{C}\otimes 1)^\vee,
\end{equation}
\begin{equation}
\label{e33}
e_{3,\mathcal{C}}^{(3)}=\frac{\triangle_\mathcal{C}^{(3)}-\beta_{2,\mathcal{C}}^{(3)}\cdot 1^{\otimes 3}}
{\beta_{3,\mathcal{C}}^{(3)}-\beta_{2,\mathcal{C}}^{(3)}}
=1^{\otimes 3}-\frac{3d}{d+2}(\triangle_\mathcal{C}\otimes 1)^\vee.
\end{equation}
\end{subequations}

\noindent For each $m$ between $\left\lceil \frac{n}{2} \right\rceil$ and $n$, we denote
\begin{equation}
\label{Slambda}
\mathbb{S}_m^{(n)}\mathcal{C}_F=\Sym^n\mathcal{C}_F\cdot e_{m,\mathcal{C}}^{(n)}.
\end{equation}
We also denote
\begin{equation}
\label{Tnn}
T_n\mathcal{C}_F=\mathbb{S}_n^{(n)}\mathcal{C}_F=\Sym^n\mathcal{C}_F\cdot e_{n,\mathcal{C}}^{(n)}.
\end{equation}

\begin{proposition}
\label{decompositionProp}
Let $m$ be an integer between $\left\lceil \frac{n}{2} \right\rceil$ and $n$.
\begin{enumerate}[label={\rm(\alph*)},leftmargin=*]
\item
If $n\ge4$ and $m<n$, then
\begin{equation*}
\big(\triangle_\mathcal{C}\otimes1^{\otimes(n-2)}\big)
\big(\triangle_\mathcal{C}^{(n)}-\beta_{m,\mathcal{C}}^{(n)}\cdot1^{\otimes n}\big)
=\frac{(n-2)(n-3)}{n(n-1)}\cdot\triangle_\mathcal{C}\otimes
\big(\triangle_\mathcal{C}^{(n-2)}-\beta_{m-1,\mathcal{C}}^{(n-2)}\cdot1^{\otimes(n-2)}\big).
\end{equation*}

\item
The set $\big\{e^{(n)}_{k,\mathcal{C}}\mid k=\left\lceil \frac{n}{2} \right\rceil,\ldots,n\big\}$ 
is a complete set of orthogonal central idempotents of $\Sym^n\mathcal{C}_F$, i.e.
$e^{(n)}_{k,\mathcal{C}}$ commutes and associates with all elements of $\Sym^n\mathcal{C}_F$,
$e^{(n)}_{k,\mathcal{C}}e^{(n)}_{\ell,\mathcal{C}}=0$ if $k\ne\ell$, 
$\big(e^{(n)}_{k,\mathcal{C}}\big)^2=e^{(n)}_{k,\mathcal{C}}$ for integers $k,\ell$ between $\left\lceil \frac{n}{2} \right\rceil$ and $n$, and
$\sum_{k=\left\lceil \frac{n}{2} \right\rceil}^ne^{(n)}_{k,\mathcal{C}}=1^{\otimes n}$.

\item
If $m<n$, then
\begin{equation}
\label{emnEq}
e_{m,\mathcal{C}}^{(n)}= 
\frac{1}{\beta_{m,\mathcal{C}}^{(n)}}\cdot\big(\triangle_\mathcal{C}\otimes e_{m-1,\mathcal{C}}^{(n-2)}\big)^\vee.
\end{equation}

\item
$e_{n,\mathcal{C}}^{(n)}\triangle_\mathcal{C}^{(n)}=0$ 
and $e_{n,\mathcal{C}}^{(n)}\equiv1^{\otimes n}\ {\rm mod}\ F[\triangle_\mathcal{C}^{(n)}]\triangle_\mathcal{C}^{(n)}$.
In particular,
\begin{equation}
\label{decompositionSym}
\Sym^n\mathcal{C}_F=T_n\mathcal{C}_F\oplus\Sym^n\mathcal{C}_F\cdot\triangle_\mathcal{C}^{(n)}\,.
\end{equation}

\item
The algebra $\mathbb{S}_m^{(n)}\mathcal{C}_F$ is isomorphic as an $F$-algebra to the
algebra $T_{2m-n}\mathcal{C}_F$ which is a nonzero central simple $F$-algebra.

\item
$\Sym^n\mathcal{C}_F=\displaystyle{\bigoplus_{m=\left\lceil \frac{n}{2} \right\rceil}^n\mathbb{S}_m^{(n)}\mathcal{C}_F}$.
\end{enumerate}
\end{proposition}

\begin{proof}
Claim (a) is \cite[Prop.\ 3.6 (a)]{Raz24}.
Claim (b) is \cite[Prop.\ 3.6 (c)]{Raz24} together with Lemma~\ref{center}.
Claim (d) is \cite[Prop.\ 3.6 (e)]{Raz24}.
Claim (e) is \cite[Prop.\ 3.9 (b) and Prop.~3.12(a)]{Raz24}.
Claim (f) is \cite[Prop.\ 3.6 (f)]{Raz24}.
Finally, Claim (c) is \cite[Prop.\ 3.6 (d)]{Raz24}, up to the calculation of the factor, which we prove now.

If $n=2$, then by \eqref{e21}, 
$e_{1,\mathcal{C}}^{(2)}=\triangle_\mathcal{C}
=\frac{1}{\beta_{1,\mathcal{C}}^{(2)}}\big(\triangle_\mathcal{C}\otimes e_{0,\mathcal{C}}^{(0)}\big)^\vee$.
If $n=3$, then by \eqref{e32},
$e_{2,\mathcal{C}}^{(3)}=\frac{3d}{d+2}\big(\triangle_\mathcal{C}\otimes1\big)^\vee
=\frac{1}{\beta_{2,\mathcal{C}}^{(3)}}\big(\triangle_\mathcal{C}\otimes e_{1,\mathcal{C}}^{(1)}\big)^\vee$.
Suppose that $n\ge4$. Then
\begin{align*}
e_{m,\mathcal{C}}^{(n)}&\stackrel{\eqref{emn}}{=}
\frac{1}{\beta_{m,\mathcal{C}}^{(n)}}\triangle_\mathcal{C}^{(n)}
\prod_{\substack{\left\lceil \frac{n}{2} \right\rceil\le m'\le n-1\\ m'\ne m}}
\frac{\triangle_\mathcal{C}^{(n)}-\beta_{m',\mathcal{C}}^{(n)}\cdot1^{\otimes n}}
{\beta_{m,\mathcal{C}}^{(n)}-\beta_{m',\mathcal{C}}^{(n)}}\\
&\stackrel{\eqref{trianglen},\eqref{eq_sym_1c}}{=}
\frac{1}{\beta_{m,\mathcal{C}}^{(n)}}\Big((\triangle_\mathcal{C}\otimes1^{\otimes(n-2)})
\prod_{\substack{\left\lceil \frac{n}{2} \right\rceil\le m'\le n-1\\ m'\ne m}}
\frac{\triangle_\mathcal{C}^{(n)}-\beta_{m',\mathcal{C}}^{(n)}\cdot1^{\otimes n}}
{\beta_{m,\mathcal{C}}^{(n)}-\beta_{m',\mathcal{C}}^{(n)}}
\Big)^\vee
\end{align*}
Thus, by Lemma~\ref{center} and \eqref{triangleSq},
\begin{align*}
e_{m,\mathcal{C}}^{(n)}&=
\frac{1}{\beta_{m,\mathcal{C}}^{(n)}}\Big(
\prod_{\substack{\left\lceil \frac{n}{2} \right\rceil\le m'\le n-1\\ m'\ne m}}
\big((\triangle_\mathcal{C}\otimes1^{\otimes(n-2)})
\frac{\triangle_\mathcal{C}^{(n)}-\beta_{m',\mathcal{C}}^{(n)}\cdot1^{\otimes n}}
{\beta_{m,\mathcal{C}}^{(n)}-\beta_{m',\mathcal{C}}^{(n)}}
\big)\Big)^\vee\\
&\stackrel{\text{(a)}}{=}
\frac{1}{\beta_{m,\mathcal{C}}^{(n)}}\Big(
\prod_{\substack{\left\lceil \frac{n}{2} \right\rceil\le m'\le n-1\\ m'\ne m}}
\big(\frac{(n-2)(n-3)}{n(n-1)}\cdot\triangle_\mathcal{C}\otimes
\frac{\triangle_\mathcal{C}^{(n-2)}-\beta_{m'-1,\mathcal{C}}^{(n-2)}\cdot1^{\otimes(n-2)}}
{\beta_{m,\mathcal{C}}^{(n)}-\beta_{m',\mathcal{C}}^{(n)}}\big)\Big)^\vee\\
&\stackrel{\eqref{betaIdentity}}{=}
\frac{1}{\beta_{m,\mathcal{C}}^{(n)}}\Big(\triangle_\mathcal{C}\otimes
\prod_{\substack{\left\lceil \frac{n-2}{2} \right\rceil\le m'-1\le n-2\\ m'-1\ne m-1}}
\frac{\triangle_\mathcal{C}^{(n-2)}-\beta_{m'-1,\mathcal{C}}^{(n-2)}\cdot1^{\otimes(n-2)}}
{\beta_{m-1,\mathcal{C}}^{(n-2)}-\beta_{m'-1,\mathcal{C}}^{(n-2)}}\Big)^\vee\\
&\stackrel{\eqref{emn}}{=}
\frac{1}{\beta_{m,\mathcal{C}}^{(n)}}\Big(\triangle_\mathcal{C}\otimes e_{m-1,\mathcal{C}}^{(n-2)}\Big)^\vee,
\end{align*}
as claimed.
\end{proof}

\begin{corollary}
\label{emnHO}
Let $m$ be an integer between $\left\lceil \frac{n}{2} \right\rceil$ and $n-1$. Then
\begin{equation}
\label{emnH}
e_{m,\mathbb{H}}^{(n)}
=
\frac{2m-n+1}{m+1}\binom{n}{m}\cdot
\big(\triangle_\mathbb{H}^{\otimes(n-m)}\otimes e_{2m-n,\mathbb{H}}^{(2m-n)}\big)^\vee,
\end{equation}
\begin{equation}
\label{emnO}
e_{m,\mathbb{O}}^{(n)}
=
2^{n-m}
\frac{(2m-n+1)(2m-n+2)(2m-n+3)}
{(m+1)(m+2)(m+3)}
\binom{n}{m}\cdot
\big(\triangle_\mathbb{O}^{\otimes(n-m)}\otimes e_{2m-n,\mathbb{O}}^{(2m-n)}\big)^\vee.
\end{equation}
\end{corollary}

\begin{proof}
Note that $m-k=n-2k$ if and only if $k=n-m$. For this $k$, we have that $m-k=n-2k=2m-n$.
Therefore, by \eqref{emnEq}
\begin{equation}
\label{emnForm}
\begin{aligned}
e_{m,\mathcal{C}}^{(n)}&=
\frac{1}{\beta_{m,\mathcal{C}}^{(n)}}\cdot
\big(\triangle_\mathcal{C}\otimes e_{m-1,\mathcal{C}}^{(n-2)}\big)^\vee
=
\frac{1}{\beta_{m,\mathcal{C}}^{(n)}\beta_{m-1,\mathcal{C}}^{(n-2)}}\cdot
\big(\triangle_\mathcal{C}^{\otimes 2}\otimes e_{m-2,\mathcal{C}}^{(n-4)}\big)^\vee\\
&=\ldots=
\frac{1}{\beta_{m,\mathcal{C}}^{(n)}\beta_{m-1,\mathcal{C}}^{(n-2)}\cdots\beta_{2m-n+1,\mathcal{C}}^{(2m-n+2)}}\cdot
\big(\triangle_\mathcal{C}^{\otimes(n-m)}\otimes e_{2m-n,\mathcal{C}}^{(2m-n)}\big)^\vee.
\end{aligned}
\end{equation}
By \eqref{beta},
\begin{equation}
\label{betaForm}
\begin{aligned}
&\frac{1}{\beta_{m,\mathcal{C}}^{(n)}\beta_{m-1,\mathcal{C}}^{(n-2)}\cdots\beta_{2m-n+1,\mathcal{C}}^{(2m-n+2)}}
=\left(\frac{d}{4}\right)^{n-m}\cdot\\
&\frac{n(n-1)(n-2)(n-3)\cdots(2m-n+2)(2m-n+1)}
{(n-m)(n-m-1)\cdots1\cdot\left(m+\frac{d}{2}-1\right)\left(m+\frac{d}{2}-2\right)\cdots\left(2m-n+\frac{d}{2}\right)}\\
&=\left(\frac{d}{4}\right)^{n-m}
\frac{m!}{\left(m+\frac{d}{2}-1\right)!}
\frac{\left(2m-n+\frac{d}{2}-1\right)!}{(2m-n)!}
\binom{n}{m}.
\end{aligned}
\end{equation}
Putting \eqref{betaForm} in \eqref{emnForm} gives \eqref{emnH} (resp.\ \eqref{emnO}) for $d=4$ (resp.\ $d=8$).
\end{proof}

\begin{proposition}[{\cite[Thm.\ 5.2]{Raz24}}]
\label{matrixProp}
Let $n$ be a positive integer. 
\begin{enumerate}[label={\rm(\alph*)},leftmargin=*]
\item
If $n$ is even, then $T_n\mathbb{H}$ is isomorphic to the algebra $M_{n+1}(\mathbb{R})$
of $n+1\times n+1$ matrices with entries in $\mathbb{R}$.

\item
If $n$ is odd, then $T_n\mathbb{H}$ is isomorphic to the algebra
$M_{\frac{n+1}{2}}(\mathbb{H})$
of $\frac{n+1}{2}\times\frac{n+1}{2}$ matrices with entries in $\mathbb{H}$.
\end{enumerate}
In particular, $T_n\mathbb{H}\otimes_\mathbb{R}\mathbb{C}\cong M_{n+1}(\mathbb{C})$.
\end{proposition}

\begin{proposition}
\label{trianglel}
Let $\ell$ be an integer between $0$ and $\left\lfloor \frac{n}{2} \right\rfloor$. Then,
\begin{equation*}
\big(1^{\otimes(n-2\ell)}\otimes\triangle_\mathcal{C}^{\otimes\ell}\big)^\vee
\in\sum_{m=\left\lceil \frac{n}{2} \right\rceil}^{n-\ell}\mathbb{Q}^\times\cdot e_{m,\mathcal{C}}^{(n)}.
\end{equation*}
\end{proposition}

\begin{proof}
This is \cite[Cor.\ 1.8]{Raz25}, applied to
the quaternion algebra $\mathbb{Q}[e_1,e_2,e_3]$ or to the octonion algebra
$\mathbb{Q}[e_1,e_2,e_3,e_4,e_5,e_6,e_7]$ over the field of rational numbers $\mathbb{Q}$.
\end{proof}

Let $m$ be an integer between $\left\lceil \frac{n}{2} \right\rceil$ and $n$.
By Proposition~\ref{decompositionProp}(b), for $\mathcal{C}=\mathbb{O}$, $e_{m,\mathbb{O}}^{(n)}$ is an idempotent that
commutes and associates with all elements of $\Sym^n\mathbb{O}$.
In particular, $\Sym^n\mathbb{H}\cdot e_{m,\mathbb{O}}^{(n)}$ is an associative algebra 
with unit $e_{m,\mathbb{O}}^{(n)}$.
By Proposition~\ref{decompositionProp}(f), for $\mathcal{C}=\mathbb{H}$,
\begin{equation}
\label{decompositionH}
\Sym^n\mathbb{H}=\bigoplus_{\ell=\left\lceil \frac{n}{2} \right\rceil}^m\mathbb{S}_\ell^{(n)}\mathbb{H}\oplus
\bigoplus_{\ell=m+1}^n\mathbb{S}_\ell^{(n)}\mathbb{H}.
\end{equation}
By Proposition~\ref{decompositionProp}(b), for $\mathcal{C}=\mathbb{H}$, 
$\sum_{\ell=\left\lceil \frac{n}{2} \right\rceil}^m e_{\ell,\mathbb{H}}^{(n)}$ is an idempotent,
which we identify as the unit of the algebra $\bigoplus_{\ell=\left\lceil \frac{n}{2} \right\rceil}^m\mathbb{S}_\ell^{(n)}\mathbb{H}$.

\begin{proposition}
\label{isom}
Let $m$ be an integer between $\left\lceil \frac{n}{2} \right\rceil$ and $n$. Then, the kernel of the epimorphism
$\varphi\colon\Sym^n\mathbb{H}\rightarrow\Sym^n\mathbb{H}\cdot e_{m,\mathbb{O}}^{(n)}$ 
given by $r\mapsto r\cdot e_{m,\mathbb{O}}^{(n)}$,
for $r\in\Sym^n\mathbb{H}$, is $\bigoplus_{\ell=m+1}^n\mathbb{S}_\ell^{(n)}\mathbb{H}$.
Thus, the restriction of $\varphi$ to $\bigoplus_{\ell=\left\lceil \frac{n}{2} \right\rceil}^m\mathbb{S}_\ell^{(n)}\mathbb{H}$
is an isomorphism
\begin{equation}
\label{isomEq}
\bigoplus_{\ell=\left\lceil \frac{n}{2} \right\rceil}^m\mathbb{S}_\ell^{(n)}\mathbb{H}\cong\Sym^n\mathbb{H}\cdot e_{m,\mathbb{O}}^{(n)}
\end{equation}
of $\mathbb{R}$-algebras.
\end{proposition}

\begin{proof}
Let $r\in\Sym^n\mathbb{H}$. By Lemma~\ref{span}, $r=\sum_{i=1}^k\alpha_i x_i^{\otimes n}$ 
for some $x_i\in\mathbb{H}$ and $\alpha_i\in\mathbb{R}$.
Set
\begin{equation*}
r'=\sum_{i=1}^k\alpha_iN(x_i)^{n-m}x_i^{\otimes(2m-n)}\in\Sym^{2m-n}\mathbb{H}.
\end{equation*}
By \eqref{emnO} and Lemma~\ref{triangleLemma}, $r\cdot e_{m,\mathbb{O}}^{(n)}=0$ if and only if
\begin{equation*}
r\cdot\big(\triangle_\mathbb{O}^{\otimes(n-m)}\otimes e_{2m-n,\mathbb{O}}^{(2m-n)}\big)^\vee
=\big(\triangle_\mathbb{O}^{\otimes(n-m)}\otimes r'e_{2m-n,\mathbb{O}}^{(2m-n)}\big)^\vee=0.
\end{equation*}
By Lemma~\ref{zeroCond}, applied to $\mathcal{C}=\mathbb{O}$, this happens if and only if $r'e_{2m-n,\mathbb{O}}^{(2m-n)}=0$.
By Proposition~\ref{decompositionProp}(d), applied to $\mathcal{C}=\mathbb{O}$ and $2m-n$ instead of $n$,
this condition holds if and only if $r'\in\Sym^{2m-n}\mathbb{O}\cdot\triangle_\mathbb{O}^{(2m-n)}$.
By Lemma~\ref{zeroIntersection}(b), applied to $2m-n$ instead of $n$,
\begin{equation*}
\Sym^{2m-n}\mathbb{H}\cap\Sym^{2m-n}\mathbb{O}\cdot\triangle_\mathbb{O}^{(2m-n)}=\{0\}.
\end{equation*}
Hence, $r\cdot e_{m,\mathbb{O}}^{(n)}=0$ if and only if $r'=0$.

Similarly, by Lemma~\ref{triangleLemma} and Lemma~\ref{zeroCond}, applied to $\mathcal{C}=\mathbb{H}$,
\begin{equation*}
r\cdot\big(\triangle_\mathbb{H}^{\otimes(n-m)}\otimes1^{\otimes(2m-n)}\big)^\vee
=\big(\triangle_\mathbb{H}^{\otimes(n-m)}\otimes r'\big)^\vee=0
\end{equation*}
if and only if $r'=0$.

Thus, $r\in\Ker(\varphi)$ if and only if 
$r\cdot\big(\triangle_\mathbb{H}^{\otimes(n-m)}\otimes1^{\otimes(2m-n)}\big)^\vee=0$.
By Proposition~\ref{trianglel}, applied to $\mathcal{C}=\mathbb{H}$ and $\ell=n-m$,
\begin{equation*}
\big(\triangle_\mathbb{H}^{\otimes(n-m)}\otimes1^{\otimes(2m-n)}\big)^\vee
=\sum_{\ell=\left\lceil \frac{n}{2} \right\rceil}^m \alpha_\ell e_{\ell,\mathbb{H}}^{(n)},
\end{equation*}
with nonzero rational numbers $\alpha_\ell$'s.
It follows from \eqref{decompositionH} that $r\in\Ker(\varphi)$ if and only if 
$r\in\bigoplus_{\ell=m+1}^n\mathbb{S}_\ell^{(n)}\mathbb{H}$, as claimed.
\end{proof}

Note that by \eqref{decompositionH},
\begin{equation*}
\Sym^n\mathbb{H}\cdot e_{m,\mathbb{O}}^{(n)}=
\bigoplus_{\ell=\left\lceil \frac{n}{2} \right\rceil}^m\Sym^n\mathbb{H}\cdot e_{\ell,\mathbb{H}}^{(n)}e_{m,\mathbb{O}}^{(n)}
\oplus
\bigoplus_{\ell=m+1}^n\Sym^n\mathbb{H}\cdot e_{\ell,\mathbb{H}}^{(n)}e_{m,\mathbb{O}}^{(n)}.
\end{equation*}
Hence, the following result is a corollary of Proposition~\ref{isom}.

\begin{corollary}
\label{MainLemma}
If $\left\lceil \frac{n}{2} \right\rceil\le m<\ell\le n$, then $e_{\ell,\mathbb{H}}^{(n)}e_{m,\mathbb{O}}^{(n)}=0$.

If $\left\lceil \frac{n}{2} \right\rceil\le\ell\le m\le n$, then the map
\begin{equation*}
\Sym^n\mathbb{H}\cdot e_{\ell,\mathbb{H}}^{(n)}\rightarrow\Sym^n\mathbb{H}\cdot e_{\ell,\mathbb{H}}^{(n)}e_{m,\mathbb{O}}^{(n)}
\end{equation*}
given by
\begin{equation*}
r\cdot e_{\ell,\mathbb{H}}^{(n)}\mapsto r\cdot e_{\ell,\mathbb{H}}^{(n)}e_{m,\mathbb{O}}^{(n)},
\end{equation*}
for $r\in\Sym^n\mathbb{H}$, is an isomorphism of $\mathbb{R}$-algebras, with units
$e_{\ell,\mathbb{H}}^{(n)}$ and $e_{\ell,\mathbb{H}}^{(n)}e_{m,\mathbb{O}}^{(n)}$, respectively.
\end{corollary}

\section[\texorpdfstring{The algebra $\Sym^n\mathbb{H}$}
                          {The algebra Sym n(H)}]
        {The algebra $\Sym^n\mathbb{H}$}

In this section we consider the algebra
\begin{equation*}
\Sym^n\mathbb{H}=\bigoplus_{\ell=\left\lceil \frac{n}{2} \right\rceil}^n\mathbb{S}_\ell^{(n)}\mathbb{H}.
\end{equation*}
We first find a complete set of primitive orthogonal idempotents in 
$\mathbb{S}_\ell^{(n)}\mathbb{H}_\mathbb{C}=\mathbb{S}_\ell^{(n)}\mathbb{H}\otimes_\mathbb{R}\mathbb{C}$ 
for each integer $\ell$ between $\left\lceil \frac{n}{2} \right\rceil$ and $n$.
From this we deduce a complete set of primitive orthogonal idempotents in $\mathbb{S}_\ell^{(n)}\mathbb{H}$ for each $\ell$.
Then we get as a corollary complete sets of primitive orthogonal idempotents in 
$\Sym^n\mathbb{H}_\mathbb{C}=\Sym^n\mathbb{H}\otimes_\mathbb{R}\mathbb{C}$ and in $\Sym^n\mathbb{H}$.
We use the following notation.

\begin{notation}
\label{aac}
For $\alpha=\alpha_1+\alpha_2\sqrt{-1}\in\mathbb{C}$, with $\alpha_1,\alpha_2\in\mathbb{R}$, we write
$\alpha^c=\alpha_1-\alpha_2\sqrt{-1}$ for its complex conjugate.
We extend this complex conjugation to 
$\Sym^n\mathbb{H}_\mathbb{C}$ by setting 
$(x+\sqrt{-1}y)^c=x-\sqrt{-1}y$, for $x,y\in\Sym^n\mathbb{H}$.
In particular, $z+z^c\in\Sym^n\mathbb{H}$ for each $z\in\Sym^n\mathbb{H}_\mathbb{C}$.

We denote
\begin{equation}
\label{a}
a:=\frac{1}{2}(1+\sqrt{-1}e_1)\in\mathbb{C}[e_1].
\end{equation}
Then
\begin{equation}
\label{ac}
a^c=\frac{1}{2}(1-\sqrt{-1}e_1).
\end{equation}
In particular,
\begin{equation}
\label{ae2}
ae_2=e_2a^c\text{ and }a^ce_2=e_2a.
\end{equation}
The elements $a$ and $a^c$ of $\mathbb{C}[e_1]$ satisfy
\begin{equation}
\label{aorth}
a^2=a,\ (a^c)^2=a^c,\ aa^c=0,\text{ and }a+a^c=1.
\end{equation}
In particular, the set $\{a,a^c\}$ is a complete set of primitive orthogonal idempotents in $\mathbb{C}[e_1]$.
Note that
\begin{equation*}
N(a)=\big(\frac{1}{2}\big)^2+\big(\frac{1}{2}\sqrt{-1}\big)^2=0
\text{ and }
N(a^c)=\big(\frac{1}{2}\big)^2+\big(-\frac{1}{2}\sqrt{-1}\big)^2=0.
\end{equation*}

We also denote
\begin{equation}
\label{square}
\square:=\triangle_{\mathbb{R}[e_1]}=\frac{1}{2}(1\otimes1+e_1\otimes e_1)\in\Sym^2\mathbb{R}[e_1].
\end{equation}
Note that
\begin{equation}
\label{saac}
\square=2(a\otimes a^c)^\vee.
\end{equation}
and, since $1\otimes1=(a+a^c)\otimes(a+a^c)$,
\begin{equation}
\label{s1}
1\otimes1-\square=a\otimes a+a^c\otimes a^c.
\end{equation}
Also, by \eqref{triangleReH},
\begin{equation}
\label{stH}
\square\triangle_\mathbb{H}=\triangle_\mathbb{H}.
\end{equation}
By \eqref{triangleSq}, applied to $\mathcal{C}=\mathbb{R}[e_1]$,
\begin{equation}
\label{ssq}
\square^2=\square.
\end{equation}
By Lemma~\ref{triangleLemma}, applied to $\mathcal{C}=\mathbb{R}[e_1]$,
\begin{equation}
\label{aas0}
(a\otimes a)\square=N(a)\square=0\text{ and }
(a^c\otimes a^c)\square=N(a^c)\square=0.
\end{equation}
Similarly,
\begin{equation}
\label{aasH}
(a\otimes a)\triangle_\mathbb{H}=0\text{ and }
(a^c\otimes a^c)\triangle_\mathbb{H}=0,
\end{equation}
by Lemma~\ref{triangleLemma}, applied to $\mathcal{C}=\mathbb{H}$.
\end{notation}

Note that $\dim_\mathbb{C}\Sym^n\mathbb{C}[e_1]=n+1$ since
the set $\big\{(e_0^{\otimes k_0}\otimes e_1^{\otimes k_1})^\vee\mid k_0+k_1=n\big\}$
forms a basis of $\Sym^n\mathbb{R}[e_1]$ over $\mathbb{R}$.

Denote $a_0:=a$ and $a_1:=a^c$.
For an integer $k$ between $0$ and $n$, we denote
\[I_k:=\big\{(i_1,\ldots,i_n)\in\{0,1\}^n\mid i_1+\ldots+i_n=n-k\big\}.\] Then
\begin{equation}
\label{akacnk}
\binom{n}{k}\big(a^{\otimes k}\otimes(a^c)^{\otimes(n-k)}\big)^\vee
=\sum_{\mathbf{i}\in I_k}a_{i_1}\otimes\cdots\otimes a_{i_n}
=\sum_{\mathbf{i}\in I_k}\sigma\big(a_{i_1}\otimes\cdots\otimes a_{i_n}\big),
\end{equation}
for each $\sigma\in\mathfrak{S}_n$, since this is an element of $\Sym^n\mathbb{C}[e_1]$.

\begin{lemma}
\label{sCe}
Let $n$ be a positive integer.
The set
\begin{align}
\label{sCeEq}
\Big\{\binom{n}{k}\big(a^{\otimes k}\otimes(a^c)^{\otimes(n-k)}\big)^\vee\mid
k=0,\ldots,n\Big\}
\end{align}
is a complete set of primitive orthogonal idempotents in $\Sym^n\mathbb{C}[e_1]$.
\end{lemma}

\begin{proof}
Let $k,\ell$ be integers between $0$ and $n$. Then
\begin{equation*}
\begin{aligned}
&\binom{n}{k}\big(a^{\otimes k}\otimes(a^c)^{\otimes(n-k)}\big)^\vee\cdot
\binom{n}{\ell}\big(a^{\otimes \ell}\otimes(a^c)^{\otimes(n-\ell)}\big)^\vee\\
&\stackrel{\eqref{eq_sym_1c}}{=}\binom{n}{k}\binom{n}{\ell}
\Big(\big(a^{\otimes k}\otimes(a^c)^{\otimes(n-k)}\big)^\vee
\big(a^{\otimes \ell}\otimes(a^c)^{\otimes(n-\ell)}\big)\Big)^\vee\\
&\stackrel{\eqref{akacnk}}{=}\binom{n}{\ell}
\sum_{\mathbf{i}\in I_k}\Big(\big(a_{i_1}\otimes\cdots\otimes a_{i_n}\big)
\big(a^{\otimes \ell}\otimes(a^c)^{\otimes(n-\ell)}\big)\Big)^\vee\\
&\stackrel{\eqref{aorth}}{=}\binom{n}{\ell}
\sum_{\mathbf{i}\in I_k}\delta_{i_1,0}\cdots\delta_{i_\ell,0}\delta_{i_{\ell+1},1}\cdots\delta_{i_n,1}
\cdot\big(a^{\otimes\ell}\otimes(a^c)^{\otimes(n-\ell)}\big)^\vee\\
&=\delta_{k\ell}\binom{n}{\ell}\cdot\big(a^{\otimes\ell}\otimes(a^c)^{\otimes(n-\ell)}\big)^\vee.
\end{aligned}
\end{equation*}
Hence, the elements in the set \eqref{sCeEq} are orthogonal idempotents.
They are not zero since, by \eqref{a} and \eqref{ac},
\begin{equation*}
\big(a^{\otimes k}\otimes(a^c)^{\otimes(n-k)}\big)^\vee=\frac{1}{2^n}\cdot1^{\otimes n}+\ldots,
\end{equation*}
for each $k$.
Since $a+a^c=1$,
\begin{equation*}
1^{\otimes n}=(a+a^c)^{\otimes n}=\sum_{k=0}^n
\binom{n}{k}\big(a^{\otimes k}\otimes(a^c)^{\otimes(n-k)}\big)^\vee.
\end{equation*}
Therefore, the set \eqref{sCeEq} is complete.
Since the number of elements in the set \eqref{sCeEq} is $n+1=\dim_\mathbb{C}\Sym^n\mathbb{C}[e_1]$,
these elements are primitive.
\end{proof}

\begin{lemma}
\label{isomCe}
Let $n$ be a positive integer. Then the map $r\mapsto r\cdot e_{n,\mathbb{H}}^{(n)}$, for $r\in\Sym^n\mathbb{C}[e_1]$,
gives an isomorphism
\begin{equation}
\label{isomCeEq}
\Sym^n\mathbb{C}[e_1]\cong\Sym^n\mathbb{C}[e_1]\cdot e_{n,\mathbb{H}}^{(n)}
\end{equation}
of $\mathbb{C}$-algebras.
\end{lemma}

\begin{proof}
The case $n=1$ is trivial, so assume $n\ge2$.
Since this map is clearly an epimorphism, we need only to show that its kernel is zero.

Indeed, let $r\in\Sym^n\mathbb{C}[e_1]$ be such that $r\cdot e_{n,\mathbb{H}}^{(n)}=0$.
By Proposition~\ref{decompositionProp}(d), applied to $\mathcal{C}=\mathbb{H}$,
$r\in\Sym^n\mathbb{H}_\mathbb{C}\cdot\triangle_\mathbb{H}^{(n)}$.
Write $r=r_1+\sqrt{-1}r_2$ with $r_1,r_2\in\Sym^n\mathbb{R}[e_1]$. Then
$r_1,r_2\in\Sym^n\mathbb{H}\cdot\triangle_\mathbb{H}^{(n)}$.
Hence, by Lemma~\ref{zeroIntersection}(a), $r_1=r_2=0$, so $r=0$, as claimed.
\end{proof}

We shall also need the formula for $x^n+y^n$ expressed as a polynomial in terms of the 
elementary symmetric polynomials $x+y$ and $xy$, which is given by Waring's formula:
\begin{lemma}[see {\cite[Eq.\ (1)]{Gou99}}]
\label{Waring}
\begin{equation}
\label{WaringEq}
x^n+y^n=\sum_{k=0}^{\left\lfloor \frac{n}{2} \right\rfloor}(-1)^k\frac{n}{n-k}\binom{n-k}{k}(x+y)^{n-2k}(xy)^k.
\end{equation}
\end{lemma}

We are ready to prove the following theorem.
\begin{theorem}
\label{Theorem1}
Let $\ell$ be an integer between $\left\lceil \frac{n}{2} \right\rceil$ and $n$.
\begin{enumerate}[label={\rm(\alph*)},leftmargin=*]
\item
The set
\begin{equation}
\label{Thm1a}
\Big\{\binom{n}{k}\big(a^{\otimes k}\otimes(a^c)^{\otimes(n-k)}\big)^\vee e^{(n)}_{\ell,\mathbb{H}}\mid
k\in\{n-\ell,\ldots,\ell\}\Big\}
\end{equation}
is a complete set of $2\ell-n+1$ primitive orthogonal idempotents in $\mathbb{S}_\ell^{(n)}\mathbb{H}\otimes_\mathbb{R}\mathbb{C}$.

\item
If $n$ is odd, then
\begin{equation}
\label{Thm1b}
\binom{n}{k}\frac{n-2k}{2^k}
\sum_{i=0}^{\left\lfloor \frac{n-1}{2} \right\rfloor-k}
\left(-\frac{1}{2}\right)^i
\frac{\binom{n-2k-i}{i}}{n-2k-i}
\big(1^{\otimes(n-2k-2i)}\otimes\square^{\otimes(k+i)}\big)^\vee\cdot e^{(n)}_{\ell,\mathbb{H}}\,,
\end{equation}
$k\in\{n-\ell,\ldots,\left\lfloor \frac{n-1}{2} \right\rfloor\}$,
form a complete set of $\frac{2\ell-n+1}{2}$ primitive orthogonal idempotents in $\mathbb{S}_\ell^{(n)}\mathbb{H}$.

\item
If $n$ is even, then
\[
\left(\frac{1}{2}\right)^{\frac{n}{2}}\binom{n}{\frac{n}{2}}
\big(\square^{\otimes\frac{n}{2}}\big)^\vee\cdot e_{\ell,\mathbb{H}}^{(n)}
\]
together with
\begin{equation}
\label{Thm1c}
\begin{aligned}
\binom{n}{k}\frac{\frac{n}{2}-k}{2^k}
\sum_{i=0}^{\left\lfloor \frac{n}{4}-\frac{k}{2}\right\rfloor}
\left(-\frac{1}{4}\right)^i
\frac{\binom{\frac{n}{2}-k-i}{i}}{\frac{n}{2}-k-i}
&\big(\square^{\otimes(k+2i)}\otimes(1\otimes1-\square)^{\otimes(\frac{n}{2}-k-2i)}\big)^\vee\\
&\cdot\frac{1}{2}\big(1^{\otimes n}+\delta\cdot e_2^{\otimes n}\big)\cdot e^{(n)}_{\ell,\mathbb{H}}\,,
\end{aligned}
\end{equation}
$k\in\{n-\ell,\ldots,\frac{n}{2}-1\}$, $\delta\in\{-1,1\}$,
if $\ell>\left\lceil \frac{n}{2} \right\rceil$,
form a complete set of $2\ell-n+1$ primitive orthogonal idempotents in $\mathbb{S}_\ell^{(n)}\mathbb{H}$.
\end{enumerate}
\end{theorem}

\begin{proof}
Suppose first that $\ell=n$. In this case it follows from Lemma~\ref{isomCe} and Lemma~\ref{sCe}
that the set
\[
\left\{\binom{n}{k}\big(a^{\otimes k}\otimes(a^c)^{\otimes(n-k)}\big)^\vee e^{(n)}_{\ell,\mathbb{H}}\mid
k=0,\ldots,n\right\}
\]
is a complete set of primitive orthogonal idempotents in 
$\Sym^n\mathbb{C}[e_1]\cdot e_{n,\mathbb{H}}^{(n)}$, which is contained in
$\Sym^n\mathbb{H}_\mathbb{C}\cdot e_{n,\mathbb{H}}^{(n)}=T_n\mathbb{H}_\mathbb{C}$.
By Proposition~\ref{matrixProp}, $T_n\mathbb{H}_\mathbb{C}\cong M_{n+1}(\mathbb{C})$. Hence, by Example~\ref{matrix}, 
the above set is also a complete set of primitive orthogonal idempotents in $T_n\mathbb{H}_\mathbb{C}:=\mathbb{S}^{(n)}_n\mathbb{H}_\mathbb{C}$, 
since the number of its elements is $n+1$.

Assume then that $n\ge2$ and $\ell<n$.

By \eqref{emnH},
\begin{equation*}
e_{\ell,\mathbb{H}}^{(n)}=\frac{2\ell-n+1}{\ell+1}\binom{n}{\ell}\cdot
\big(\triangle_\mathbb{H}^{\otimes(n-\ell)}\otimes e_{2\ell-n,\mathbb{H}}^{(2\ell-n)}\big)^\vee.
\end{equation*}
Let $k$ be an integer between $0$ and $n$. Then,
as $\big(a^{\otimes k}\otimes(a^c)^{\otimes(n-k)}\big)^\vee\in\Sym^n\mathbb{H}_\mathbb{C}$,
\begin{equation}
\label{summ}
\begin{aligned}
&\binom{n}{k}\big(a^{\otimes k}\otimes(a^c)^{\otimes(n-k)}\big)^\vee
\big(\triangle_\mathbb{H}^{\otimes(n-\ell)}\otimes e_{2\ell-n,\mathbb{H}}^{(2\ell-n)}\big)^\vee\\
&\stackrel{\eqref{eq_sym_1c}}{=}
\binom{n}{k}\Big(\big(a^{\otimes k}\otimes(a^c)^{\otimes(n-k)}\big)^\vee
\big(\triangle_\mathbb{H}^{\otimes(n-\ell)}\otimes e_{2\ell-n,\mathbb{H}}^{(2\ell-n)}\big)\Big)^\vee\\
&\stackrel{\eqref{akacnk}}{=}\sum_{\mathbf{i}\in I_k}
\Big(\big((a_{i_1}\otimes a_{i_2})^\vee\otimes\cdots\otimes 
(a_{i_{2(n-\ell)-1}}\otimes a_{i_{2(n-\ell)}})^\vee\otimes
a_{i_{2(n-\ell)+1}}\otimes\cdots\otimes a_{i_n}\big)\cdot\\
&\qquad\qquad\quad
\big(\triangle_\mathbb{H}^{\otimes(n-\ell)}\otimes e_{2\ell-n,\mathbb{H}}^{(2\ell-n)}\big)\Big)^\vee,
\end{aligned}
\end{equation}
since
\begin{align*}
&(a_{i_1}\otimes a_{i_2})^\vee\otimes\cdots\otimes 
(a_{i_{2(n-\ell)-1}}\otimes a_{i_{2(n-\ell)}})^\vee\otimes
a_{i_{i_{2(n-\ell)+1}}}\otimes\cdots\otimes a_{i_n}\\
&=\frac{1}{2^{n-\ell}}\sum_{\sigma_1\in\{{\rm id},(1\ 2)\}}\cdots
\sum_{\sigma_{n-\ell}\in\{{\rm id},(2(n-\ell)-1\ 2(n-\ell))\}}
(\sigma_1\cdots\sigma_{n-\ell})\big(a_{i_1}\otimes\cdots\otimes a_{i_n}\big),
\end{align*}
where, for $i\ne j$ between $1$ and $n$, $(i\ j)$ is the transposition in $\mathfrak{S}_n$ that swaps
$i$ and $j$ while leaving all other $n-2$ elements in the set $\{1,\ldots,n\}$ fixed. 
By \eqref{aasH}, \[(a\otimes a)\triangle_\mathbb{H}=0=(a^c\otimes a^c)\triangle_\mathbb{H}.\]
Also, by \eqref{saac} and \eqref{stH},
$2(a\otimes a^c)^\vee\triangle_\mathbb{H}=\triangle_\mathbb{H}$.
Therefore, only summands of the form
\begin{equation*}
(a\otimes a^c)^\vee\otimes\cdots\otimes 
(a\otimes a^c)^\vee\otimes
a_{i_{2(n-\ell)+1}}\otimes\cdots\otimes a_{i_n}
\end{equation*}
do not necessarily annihilate 
$\triangle_\mathbb{H}^{\otimes(n-\ell)}\otimes e_{2\ell-n,\mathbb{H}}^{(2\ell-n)}$.
Hence, if $k$ is a nonnegative integer such that $k<n-\ell$ or $n-k<n-\ell$, 
then each of the summands
\begin{equation*}
(a_{i_1}\otimes a_{i_2})^\vee\otimes\cdots\otimes 
(a_{i_{2(n-\ell)-1}}\otimes a_{i_{2(n-\ell)}})^\vee\otimes
a_{i_{2(n-\ell)+1}}\otimes\cdots\otimes a_{i_n}
\end{equation*}
annihilates $\triangle_\mathbb{H}^{\otimes(n-\ell)}\otimes e_{2\ell-n,\mathbb{H}}^{(2\ell-n)}$.
Thus, 
\begin{equation}
\label{aace0}
\begin{aligned}
&\binom{n}{k}\big(a^{\otimes k}\otimes(a^c)^{\otimes(n-k)}\big)^\vee e^{(n)}_{\ell,\mathbb{H}}=0\\
&\text{ for each }k\in\{0,\ldots,n-\ell-1\}\cup\{\ell+1,\ldots,n\}.
\end{aligned}
\end{equation}

Now let $k$ be an integer between $n-\ell$ and $\ell$. Then
\begin{equation*}
\begin{aligned}
&\binom{n}{k}\big(a^{\otimes k}\otimes(a^c)^{\otimes(n-k)}\big)^\vee e^{(n)}_{\ell,\mathbb{H}}\\
&\stackrel{\eqref{emnH}}{=}
\frac{2\ell-n+1}{\ell+1}\binom{n}{\ell}\binom{n}{k}\big(a^{\otimes k}\otimes(a^c)^{\otimes(n-k)}\big)^\vee
\big(\triangle_\mathbb{H}^{\otimes(n-\ell)}\otimes e_{2\ell-n,\mathbb{H}}^{(2\ell-n)}\big)^\vee\\
&\stackrel{\eqref{summ}}{=}
\frac{2\ell-n+1}{\ell+1}\binom{n}{\ell}\binom{2\ell-n}{k-n+\ell}
\Bigg(\big(2(a\otimes a^c)^\vee\big)^{\otimes(n-\ell)}\otimes
\big(a^{\otimes(k-n+\ell)}\otimes(a^c)^{\otimes(\ell-k)}\big)^\vee\cdot\\
&\qquad\qquad\qquad\qquad\qquad\qquad\qquad\qquad
\big(\triangle_\mathbb{H}^{\otimes(n-\ell)}\otimes e_{2\ell-n,\mathbb{H}}^{(2\ell-n)}\big)\Bigg)^\vee\\
&=
\frac{2\ell-n+1}{\ell+1}\binom{n}{\ell}\binom{2\ell-n}{k-n+\ell}
\Big(\triangle_\mathbb{H}^{\otimes(n-\ell)}\otimes
\big(a^{\otimes(k-n+\ell)}\otimes(a^c)^{\otimes(\ell-k)}\big)^\vee
e_{2\ell-n,\mathbb{H}}^{(2\ell-n)}\Big)^\vee.
\end{aligned}
\end{equation*}
By the case $\ell=n$, proved in the first paragraph of the proof, applied to $2\ell-n$ instead of $\ell$ and $n$,
the set
\begin{equation*}
\Big\{\binom{2\ell-n}{k-n+\ell}\big(a^{\otimes(k-n+\ell)}\otimes(a^c)^{\otimes(\ell-k)}\big)^\vee
e_{2\ell-n,\mathbb{H}}^{(2\ell-n)}
\mid k\in\{n-\ell,\ldots,\ell\}\Big\}
\end{equation*}
is a complete set of primitive orthogonal idempotents in $T_{2\ell-n}\mathbb{H}_\mathbb{C}$.
In particular, by Lemma~\ref{zeroCond}, applied to $\mathcal{C}=\mathbb{H}$, each element in the set \eqref{Thm1a} is nonzero.
Moreover, by Lemma~\ref{sCe}, the set \eqref{Thm1a}
consists of $2\ell-n+1$ orthogonal idempotents in $\mathbb{S}_\ell^{(n)}\mathbb{H}_\mathbb{C}$.
Also, by \eqref{aace0},
\begin{equation*}
\begin{aligned}
\sum_{k=n-\ell}^\ell
\binom{n}{k}\big(a^{\otimes k}\otimes(a^c)^{\otimes(n-k)}\big)^\vee e^{(n)}_{\ell,\mathbb{H}}
&=\sum_{k=0}^n
\binom{n}{k}\big(a^{\otimes k}\otimes(a^c)^{\otimes(n-k)}\big)^\vee e^{(n)}_{\ell,\mathbb{H}}\\
&=(a+a^c)^{\otimes n}\cdot e^{(n)}_{\ell,\mathbb{H}}
\stackrel{\eqref{aorth}}{=}e^{(n)}_{\ell,\mathbb{H}},
\end{aligned}
\end{equation*}
which is the unit element in the algebra $\mathbb{S}_\ell^{(n)}\mathbb{H}_\mathbb{C}$.
Thus, the set \eqref{Thm1a} is complete.
By Proposition~\ref{decompositionProp}(e), applied to $\mathcal{C}=\mathbb{H}$, 
$\mathbb{S}_\ell^{(n)}\mathbb{H}_\mathbb{C}\cong T_{2\ell-n}\mathbb{H}_\mathbb{C}$.
Hence, 
the orthogonal idempotents in the set \eqref{Thm1a} are primitive,
since their number is $2\ell-n+1$.
This completes the proof of (a).

(b): Suppose that $n$ is odd.
By Proposition~\ref{decompositionProp}(e), applied to $\mathcal{C}=\mathbb{H}$, and Proposition~\ref{matrixProp}(b), 
$\mathbb{S}_\ell^{(n)}\mathbb{H} \cong T_{2\ell-n}\mathbb{H} \cong M_{\frac{2\ell-n+1}{2}}(\mathbb{H})$,
the ring of $\frac{2\ell-n+1}{2}\times\frac{2\ell-n+1}{2}$ matrices over the division ring $\mathbb{H}$.
Hence, by Example~\ref{matrix}, 
the number of elements of a complete set of primitive orthogonal idempotents in $\mathbb{S}_\ell^{(n)}\mathbb{H}$
is $\ell-\frac{n-1}{2}$. Thus, the elements
\begin{equation}
\label{nodd}
\binom{n}{k}\Big(\big(a^{\otimes k}\otimes(a^c)^{\otimes(n-k)}\big)^\vee
+\big(a^{\otimes(n-k)}\otimes(a^c)^{\otimes k}\big)^\vee\Big)e_{\ell,\mathbb{H}}^{(n)},
\end{equation}
for $k\in\{n-\ell,\ldots,\frac{n-1}{2}\}$,
form a complete set of primitive orthogonal idempotents in $\mathbb{S}_\ell^{(n)}\mathbb{H}$, since both
$\binom{n}{k}\big(a^{\otimes k}\otimes(a^c)^{\otimes(n-k)}\big)^\vee e_{\ell,\mathbb{H}}^{(n)}$ and
$\binom{n}{k}\big(a^{\otimes(n-k)}\otimes(a^c)^{\otimes k}\big)^\vee e_{\ell,\mathbb{H}}^{(n)}$
are elements of the set \eqref{Thm1a} and they are complex conjugates,
so their sum lies in $\mathbb{S}_\ell^{(n)}\mathbb{H}$.

Let $k$ be an integer between $n-\ell$ and $\frac{n-1}{2}$. 
In particular, $n-k\ge\frac{n+1}{2}>k$ and $2k<n$.
We shall rewrite \eqref{nodd} as \eqref{Thm1b}.
First note that, by \eqref{eq_sym_1b}, the element \eqref{nodd} is equal to
\begin{equation}
\label{nodd2}
\binom{n}{k}\Big(\big((a\otimes a^c)^\vee\big)^{\otimes k}\otimes
\big(a^{\otimes(n-2k)}+(a^c)^{\otimes(n-2k)}\big)\Big)^\vee e_{\ell,\mathbb{H}}^{(n)}.
\end{equation}
By \eqref{saac}, $(a\otimes a^c)^\vee=\frac{1}{2}\square$.
Also, by \eqref{WaringEq}, applied to the symmetric tensor product 
$\Sym^{n_1}\mathbb{C}[e_1]\times\Sym^{n_2}\mathbb{C}[e_1]\rightarrow\Sym^{n_1+n_2}\mathbb{C}[e_1]$,
defined by $(x_1,x_2)\mapsto(x_1\otimes x_2)^\vee$,
and to $x=a$, $y=a^c$,
\begin{equation*}
\begin{aligned}
&a^{\otimes(n-2k)}+(a^c)^{\otimes(n-2k)}\\
&=\sum_{i=0}^{\left\lfloor \frac{n-2k-1}{2} \right\rfloor}
(-1)^i\frac{n-2k}{n-2k-i}\binom{n-2k-i}{i}
\Big((a+a^c)^{\otimes(n-2k-2i)}\otimes\big((a\otimes a^c)^\vee\big)^{\otimes i}\Big)^\vee\\
&\stackrel{\eqref{aorth},\eqref{saac}}{=}
\sum_{i=0}^{\frac{n-1}{2}-k}
\left(-\frac{1}{2}\right)^i
\frac{n-2k}{n-2k-i}
\binom{n-2k-i}{i}
\Big(1^{\otimes(n-2k-2i)}\otimes\square^{\otimes i}\Big)^\vee.
\end{aligned}
\end{equation*}
Putting the above equations in \eqref{nodd2} gives \eqref{Thm1b}.

(c):
Suppose now that $n$ is even.
By Proposition~\ref{decompositionProp}(e), applied to $\mathcal{C}=\mathbb{H}$, and Proposition~\ref{matrixProp}(a), 
$\mathbb{S}_\ell^{(n)}\mathbb{H}\cong T_{2\ell-n}\mathbb{H}\cong M_{2\ell-n+1}(\mathbb{R})$. 
Hence, by Example~\ref{matrix}, 
the number of elements of a complete set of primitive orthogonal idempotents in $\mathbb{S}_\ell^{(n)}\mathbb{H}$
is $2\ell-n+1$. 

As above, the element
$\binom{n}{\frac{n}{2}}\big(a^{\otimes\frac{n}{2}}\otimes(a^c)^{\otimes\frac{n}{2}}\big)^\vee\cdot e_{\ell,\mathbb{H}}^{(n)}$,
which equals
$\left(\frac{1}{2}\right)^{\frac{n}{2}}\binom{n}{\frac{n}{2}}
\big(\square^{\otimes\frac{n}{2}}\big)^\vee\cdot e_{\ell,\mathbb{H}}^{(n)}$
by \eqref{saac},
together with the elements \eqref{nodd},
for $k\in\{n-\ell,\ldots,\frac{n}{2}-1\}$,
if $\ell>\left\lceil \frac{n}{2} \right\rceil$,
form a complete set of orthogonal idempotents in $\mathbb{S}_\ell^{(n)}\mathbb{H}$.
We shall prove that each of the elements \eqref{nodd}, for
$k\in\{n-\ell,\ldots,\frac{n}{2}-1\}$,
if $\ell>\left\lceil \frac{n}{2} \right\rceil$,
is a sum of two nonzero orthogonal idempotents
in $\mathbb{S}_\ell^{(n)}\mathbb{H}$,
so there is a total of
$1+2\big(\left(\frac{n}{2}-1\right)-(n-\ell)+1\big)=2\ell-n+1$
nonzero orthogonal idempotents in $\mathbb{S}_\ell^{(n)}\mathbb{H}$
that sum to $e_{\ell,\mathbb{H}}^{(n)}$.
Thus, these elements form a complete set of primitive orthogonal idempotents in $\mathbb{S}_\ell^{(n)}\mathbb{H}$.

Assume that $\ell>\left\lceil \frac{n}{2} \right\rceil$
and let $k$ be an integer between $n-\ell$ and $\frac{n}{2}-1$. 
By \eqref{ae2}, $e_2^{\otimes n}$ commutes with the element \eqref{nodd} (for this fixed $k$).
Also $\frac{1}{2}\big(1^{\otimes n}+e_2^{\otimes n}\big)$ and $\frac{1}{2}\big(1^{\otimes n}-e_2^{\otimes n}\big)$
are two orthogonal idempotents that sum to $1^{\otimes n}$. Hence,
\begin{equation}
\label{e2orth}
\binom{n}{k}\Big(\big(a^{\otimes k}\otimes(a^c)^{\otimes(n-k)}\big)^\vee
+\big(a^{\otimes(n-k)}\otimes(a^c)^{\otimes k}\big)^\vee\Big)
\cdot\frac{1}{2}\big(1^{\otimes n}+\delta\cdot e_2^{\otimes n}\big)
\cdot e_{\ell,\mathbb{H}}^{(n)},
\end{equation}
for $\delta\in\{-1,1\}$, are two orthogonal idempotents that sum to \eqref{nodd}.
We show that each of them is not zero.

Indeed, by \eqref{emnH} and Lemma~\ref{triangleLemma},
\begin{equation*}
\frac{1}{2}\big(1^{\otimes n}+\delta\cdot e_2^{\otimes n}\big)\cdot e_{\ell,\mathbb{H}}^{(n)}
=
\frac{2\ell-n+1}{\ell+1}\binom{n}{\ell}\cdot
\Big(\triangle_\mathbb{H}^{\otimes(n-\ell)}\otimes
\frac{1}{2}\big(1^{\otimes(2\ell-n)}+\delta\cdot e_2^{\otimes(2\ell-n)}\big)\cdot e_{2\ell-n,\mathbb{H}}^{(2\ell-n)}\Big)^\vee.
\end{equation*}
Since $k\ge n-\ell$ and $n-k\ge\frac{n}{2}+1>n-\ell$, multiplying the above element by 
\begin{equation*}
\binom{n}{k}\Big(\big(a^{\otimes k}\otimes(a^c)^{\otimes(n-k)}\big)^\vee
+\big(a^{\otimes(n-k)}\otimes(a^c)^{\otimes k}\big)^\vee\Big),
\end{equation*}
we get, as in the proof of (a), the element
\begin{equation*}
\begin{aligned}
\frac{2\ell-n+1}{\ell+1}\binom{n}{\ell}&\binom{2\ell-n}{k-n+\ell}\cdot\Bigg(\triangle_\mathbb{H}^{\otimes(n-\ell)}\\
&\otimes\Big(\big(a^{\otimes(k-n+\ell)}\otimes(a^c)^{\otimes(\ell-k)}\big)^\vee+
\big(a^{\otimes(\ell-k)}\otimes(a^c)^{\otimes(k-n+\ell)}\big)^\vee\Big)\\
&\cdot\frac{1}{2}\big(1^{\otimes(2\ell-n)}+\delta\cdot e_2^{\otimes(2\ell-n)}\big)
\cdot e_{2\ell-n,\mathbb{H}}^{(2\ell-n)}\Bigg)^\vee.
\end{aligned}
\end{equation*}
By Lemma~\ref{zeroCond} and Proposition~\ref{decompositionProp}(d), applied to $\mathcal{C}=\mathbb{H}$
and $2\ell-n$ instead of $n$,
this element is zero if and only if
\begin{equation*}
\begin{aligned}
r&:=\Big(\big(a^{\otimes(k-n+\ell)}\otimes(a^c)^{\otimes(\ell-k)}\big)^\vee+
\big(a^{\otimes(\ell-k)}\otimes(a^c)^{\otimes(k-n+\ell)}\big)^\vee\Big)\\
&+\delta\cdot\Big(\big((ae_2)^{\otimes(k-n+\ell)}\otimes(a^ce_2)^{\otimes(\ell-k)}\big)^\vee+
\big((ae_2)^{\otimes(\ell-k)}\otimes(a^ce_2)^{\otimes(k-n+\ell)}\big)^\vee\Big)
\end{aligned}
\end{equation*}
is an element of $\Sym^{2\ell-n}\mathbb{H}\cdot\triangle_\mathbb{H}^{(2\ell-n)}$.
By Proposition~\ref{LocalGlobal}, applied to $\mathcal{C}=\mathbb{H}$,
this happens if and only if $r_e=0$ for each $e\in\Imag_1\mathbb{H}$. 
Let $e=\alpha_1e_1+\alpha_2e_2+\alpha_3e_3\in\Imag_1\mathbb{H}$ with $\alpha_1,\alpha_2,\alpha_3\in\mathbb{R}$
satisfying $\alpha_1^2+\alpha_2^2+\alpha_3^2=1$.
By Lemma~\ref{epart},
\begin{gather*}
    a_e=\frac{1}{2}(1+\alpha_1\sqrt{-1}e), \quad
    a^c_e=\frac{1}{2}(1-\alpha_1\sqrt{-1}e), \\
    (ae_2)_e=\frac{1}{2}(e_2+\sqrt{-1}e_3)_e=\frac{1}{2}(\alpha_2+\alpha_3\sqrt{-1})e, \\
    (a^ce_2)_e=\frac{1}{2}(e_2-\sqrt{-1}e_3)_e=\frac{1}{2}(\alpha_2-\alpha_3\sqrt{-1})e.
\end{gather*}
Set $\alpha:=\alpha_2+\alpha_3\sqrt{-1}$, so $\alpha^c=\alpha_2-\alpha_3\sqrt{-1}$.
Since $2k<n$, we have $k-n+\ell<\ell-k$.
In particular, for $\alpha_1=0$, we get

\begin{equation*}
\begin{aligned}
r_e&=\frac{1}{2^{2\ell-n}}\Big(2+\delta\cdot\big(\alpha^{k-n+\ell}(\alpha^c)^{\ell-k}
+\alpha^{\ell-k}(\alpha^c)^{k-n+\ell}\big)e^{2\ell-n}\Big)\\
&=\frac{1}{2^{2\ell-n}}\Big(2+\delta\cdot(\alpha\alpha^c)^{k-n+\ell}\big((\alpha^c)^{n-2k}+\alpha^{n-2k}\big)
(-1)^{\ell-\frac{n}{2}}\Big)\\
&=\frac{1}{2^{2\ell-n}}\Big(2+\delta\cdot\big((\alpha^c)^{n-2k}+\alpha^{n-2k}\big)(-1)^{\ell-\frac{n}{2}}\Big).
\end{aligned}
\end{equation*}

By \eqref{WaringEq}, applied to $x=\alpha$ and $y=\alpha^c$, we get, using 
$\alpha+\alpha^c=2\alpha_2$ and $\alpha\alpha^c=1$,
\begin{equation*}
\alpha^{n-2k}+(\alpha^c)^{n-2k}
=
\sum_{j=0}^{\frac{n}{2}-k}
(-1)^j
\frac{n-2k}{n-2k-j}
\binom{n-2k-j}{j}
(2\alpha_2)^{n-2k-2j}.
\end{equation*}
Set
\begin{equation*}
P(X):=2+\delta\cdot(-1)^{\ell-\frac{n}{2}}
\sum_{j=0}^{\frac{n}{2}-k}
(-1)^j\frac{n-2k}{n-2k-j}\binom{n-2k-j}{j}(2X)^{n-2k-2j}.
\end{equation*}
Since this polynomial is not constant (because $2k<n$) and 
has only finitely many roots, we can find $0<\alpha_2<1$ such that $P(\alpha_2)\ne0$.
Thus, if $\alpha_3=\sqrt{1-\alpha_2^2}$ and $e=\alpha_2e_2+\alpha_3e_3$, we get $r_e\ne0$.
This proves that the element \eqref{e2orth} is not zero.

We shall rewrite \eqref{e2orth} as \eqref{Thm1c}.
First note that, by \eqref{eq_sym_1b}, the element \eqref{e2orth} is equal to
\begin{equation}
\label{neven2}
\binom{n}{k}\Big(\big((a\otimes a^c)^\vee\big)^{\otimes k}\otimes
\big(a^{\otimes(n-2k)}+(a^c)^{\otimes(n-2k)}\big)\Big)^\vee 
\cdot\frac{1}{2}\big(1^{\otimes n}+\delta\cdot e_2^{\otimes n}\big)\cdot e_{\ell,\mathbb{H}}^{(n)}.
\end{equation}
By \eqref{saac}, $(a\otimes a^c)^\vee=\frac{1}{2}\square$
and, by \eqref{s1}, $a\otimes a+a^c\otimes a^c=1\otimes1-\square$.
Also, by \eqref{WaringEq}, applied to the symmetric tensor product 
and to $x=a^{\otimes2}$, $y=(a^c)^{\otimes2}$,
\begin{equation*}
\begin{aligned}
&a^{\otimes(n-2k)}+(a^c)^{\otimes(n-2k)}=\sum_{i=0}^{\left\lfloor \frac{n}{4}-\frac{k}{2} \right\rfloor}\\
&(-1)^i
\frac{\frac{n}{2}-k}{\frac{n}{2}-k-i}
\binom{\frac{n}{2}-k-i}{i}
\Big(\big(a^{\otimes2}+(a^c)^{\otimes2}\big)^{\otimes(\frac{n}{2}-k-2i)}
\otimes\big(\big(a^{\otimes2}\otimes(a^c)^{\otimes2}\big)^\vee\big)^{\otimes i}\Big)^\vee\\
&\stackrel{\eqref{s1},\eqref{saac}}{=}
\sum_{i=0}^{\left\lfloor \frac{n}{4}-\frac{k}{2} \right\rfloor}
\left(-\frac{1}{4}\right)^i
\frac{\frac{n}{2}-k}{\frac{n}{2}-k-i}
\binom{\frac{n}{2}-k-i}{i}
\Big((1\otimes1-\square)^{\otimes(\frac{n}{2}-k-2i)}\otimes\square^{\otimes 2i}\Big)^\vee.
\end{aligned}
\end{equation*}
Putting the above equations in \eqref{neven2} gives \eqref{Thm1c}.
\end{proof}

By Proposition~\ref{decompositionProp}(f), applied to $\mathcal{C}=\mathbb{H}$,
$\Sym^n\mathbb{H}_F=\bigoplus_{\ell=\left\lceil \frac{n}{2} \right\rceil}^n\mathbb{S}_\ell^{(n)}\mathbb{H}_F$, for $F=\mathbb{R}$ or $F=\mathbb{C}$.
Note that $\sum_{\ell=\left\lceil \frac{n}{2} \right\rceil}^n(2\ell-n+1)$ is equal to
\begin{equation*}
\begin{cases}
\sum_{\ell=\frac{n}{2}}^n(2\ell-n+1)=\left(\frac{n}{2}+1\right)^2=\frac{1}{4}(n+2)^2
&\text{if }n\text{ is even}\\
\sum_{\ell=\frac{n+1}{2}}^n(2\ell-n+1)=\frac{n+1}{2}\cdot\left(\frac{n+1}{2}+1\right)=\frac{1}{4}(n+1)(n+3)
&\text{if }n\text{ is odd}.
\end{cases}
\end{equation*}
Hence, by Corollary~\ref{completeCorollary}, the following result is a corollary of Theorem~\ref{Theorem1}.

\begin{corollary}
\label{Corollary2}
\begin{enumerate}[label={\rm(\alph*)},leftmargin=*]
\item
The set
\[
\left\{
\binom{n}{k}\bigl(a^{\otimes k}\otimes(a^c)^{\otimes(n-k)}\bigr)^\vee e^{(n)}_{\ell,\mathbb{H}}
\;\middle|\;
\begin{aligned}
&k\in\{n-\ell,\ldots,\ell\},\\
&\ell\in\left\{\left\lceil \frac{n}{2} \right\rceil,\ldots,n\right\}
\end{aligned}
\right\}
\]
is a complete set of $\frac{1}{4}(n+1)(n+3)$ (resp.\ $\frac{1}{4}(n+2)^2$)
primitive orthogonal idempotents in $\Sym^n\mathbb{H}\otimes_\mathbb{R}\mathbb{C}$
if $n$ is odd (resp.\ even).

\item
If $n$ is odd, then
\begin{align*}
\binom{n}{k}\frac{n-2k}{2^k}
\sum_{i=0}^{\frac{n-1}{2}-k}
\left(-\frac{1}{2}\right)^i
\frac{\binom{n-2k-i}{i}}{n-2k-i}
\big(1^{\otimes(n-2k-2i)}\otimes\square^{\otimes(k+i)}\big)^\vee\cdot e^{(n)}_{\ell,\mathbb{H}}\,,
\end{align*}
$k\in\{n-\ell,\ldots,\frac{n-1}{2}\}$, $\ell\in\left\{\frac{n+1}{2},\ldots,n\right\}$,
form a complete set of $\frac{1}{8}(n+1)(n+3)$ primitive orthogonal idempotents in $\Sym^n\mathbb{H}$.

\item
If $n$ is even, then
\[
\left(\frac{1}{2}\right)^{\frac{n}{2}}\binom{n}{\frac{n}{2}}
\big(\square^{\otimes\frac{n}{2}}\big)^\vee\cdot e_{\frac{n}{2},\mathbb{H}}^{(n)}
\]
and
\[
\left(\frac{1}{2}\right)^{\frac{n}{2}}\binom{n}{\frac{n}{2}}
\big(\square^{\otimes\frac{n}{2}}\big)^\vee\cdot e_{\ell,\mathbb{H}}^{(n)}
\]
together with
\begin{align*}
\binom{n}{k}\frac{\frac{n}{2}-k}{2^k}
\sum_{i=0}^{\frac{n}{4}-\lfloor\frac{k}{2}\rfloor}
\left(-\frac{1}{4}\right)^i
\frac{\binom{\frac{n}{2}-k-i}{i}}{\frac{n}{2}-k-i}
&\big(\square^{\otimes(k+2i)}\otimes(1\otimes1-\square)^{\otimes(\frac{n}{2}-k-2i)}\big)^\vee\\
&\cdot\frac{1}{2}\big(1^{\otimes n}+\delta\cdot e_2^{\otimes n}\big)\cdot e^{(n)}_{\ell,\mathbb{H}}\,,
\end{align*}
$k\in\{n-\ell,\ldots,\frac{n}{2}-1\}$, $\delta\in\{-1,1\}$, $\ell\in\left\{\frac{n}{2}+1,\ldots,n\right\}$,
form a complete set of $\frac{1}{4}(n+2)^2$ primitive orthogonal idempotents in $\Sym^n\mathbb{H}$.
\end{enumerate}
\end{corollary}


\section[\texorpdfstring{An associative subalgebra of $\Sym^n\mathbb{O}$}
                          {An associative subalgebra of SymnO}]
        {An associative subalgebra of $\Sym^n\mathbb{O}$}
\label{IV}

In this section we consider the associative subalgebra 
$\Sym^n\mathbb{H}\cdot Z(\Sym^n\mathbb{O})$ of $\Sym^n\mathbb{O}$,
generated by $\Sym^n\mathbb{H}$ and the center $Z(\Sym^n\mathbb{O})$ of $\Sym^n\mathbb{O}$.
By Proposition~\ref{decompositionProp}(e),(f), applied to $\mathcal{C}=\mathbb{O}$,
$Z(\Sym^n\mathbb{O})=\bigoplus_{m=\left\lceil \frac{n}{2} \right\rceil}^n Z(\mathbb{S}^{(n)}_m\mathbb{O})
=\bigoplus_{m=\left\lceil \frac{n}{2} \right\rceil}^n\mathbb{R}\cdot e^{(n)}_{m,\mathbb{O}}$.
Thus,
\begin{equation*}
\Sym^n\mathbb{H}\cdot Z(\Sym^n\mathbb{O})=\bigoplus_{m=\left\lceil \frac{n}{2} \right\rceil}^n\Sym^n\mathbb{H}\cdot e_{m,\mathbb{O}}^{(n)}.
\end{equation*}
We first find a complete set of primitive orthogonal idempotents in $\Sym^n\mathbb{H}_\mathbb{C}\cdot e_{m,\mathbb{O}}^{(n)}$
and $\Sym^n\mathbb{H}\cdot e_{m,\mathbb{O}}^{(n)}$
for each integer $m$ between $\left\lceil \frac{n}{2} \right\rceil$ and $n$.
Then we get as a corollary complete sets of primitive orthogonal idempotents in 
$\Sym^n\mathbb{H}_\mathbb{C}\cdot Z(\Sym^n\mathbb{O})$ and in $\Sym^n\mathbb{H}\cdot Z(\Sym^n\mathbb{O})$.

\begin{theorem}
\label{Theorem3}
Let $m$ be an integer between $\left\lceil \frac{n}{2} \right\rceil$ and $n$.
\begin{enumerate}[label={\rm(\alph*)},leftmargin=*]
\item
The set
\[
\left\{
\binom{n}{k}\bigl(a^{\otimes k}\otimes(a^c)^{\otimes(n-k)}\bigr)^\vee 
e^{(n)}_{\ell,\mathbb{H}}e^{(n)}_{m,\mathbb{O}}
\;\middle|\;
\begin{aligned}
&k\in\{n-\ell,\ldots,\ell\},\\
&\ell\in\left\{\left\lceil \frac{n}{2} \right\rceil,\ldots,m\right\}
\end{aligned}
\right\}
\]
is a complete set of $\frac{1}{4}(2m-n+1)(2m-n+3)$ (resp.\ $\frac{1}{4}(2m-n+2)^2$)
primitive orthogonal idempotents in $(\Sym^n\mathbb{H}\otimes_\mathbb{R}\mathbb{C})\cdot e^{(n)}_{m,\mathbb{O}}$
if $n$ is odd (resp.\ even).

\item
If $n$ is odd, then
\begin{align*}
\binom{n}{k}\frac{n-2k}{2^k}
\sum_{i=0}^{\frac{n-1}{2}-k}
\left(-\frac{1}{2}\right)^i
\frac{\binom{n-2k-i}{i}}{n-2k-i}
\big(1^{\otimes(n-2k-2i)}\otimes\square^{\otimes(k+i)}\big)^\vee\cdot e^{(n)}_{\ell,\mathbb{H}}e^{(n)}_{m,\mathbb{O}}\,,
\end{align*}
where $k$ varies in $\{n-\ell,\ldots,\frac{n-1}{2}\}$ and $\ell$ varies in $\left\{\frac{n+1}{2},\ldots,m\right\}$,
form a complete set of $\frac{1}{8}(2m-n+1)(2m-n+3)$
primitive orthogonal idempotents in $\Sym^n\mathbb{H}\cdot e^{(n)}_{m,\mathbb{O}}$.

\item
If $n$ is even, then
\[
\left(\frac{1}{2}\right)^{\frac{n}{2}}\binom{n}{\frac{n}{2}}
\big(\square^{\otimes\frac{n}{2}}\big)^\vee\cdot e_{\frac{n}{2},\mathbb{H}}^{(n)}e_{m,\mathbb{O}}^{(n)}
\]
and
\[
\left(\frac{1}{2}\right)^{\frac{n}{2}}\binom{n}{\frac{n}{2}}
\big(\square^{\otimes\frac{n}{2}}\big)^\vee\cdot e_{\ell,\mathbb{H}}^{(n)}e_{m,\mathbb{O}}^{(n)}
\]
together with
\begin{align*}
\binom{n}{k}\frac{\frac{n}{2}-k}{2^k}
\sum_{i=0}^{\lfloor\frac{n}{4}-\frac{k}{2}\rfloor}
\left(-\frac{1}{4}\right)^i
\frac{\binom{\frac{n}{2}-k-i}{i}}{\frac{n}{2}-k-i}
&\big(\square^{\otimes(k+2i)}\otimes(1\otimes1-\square)^{\otimes(\frac{n}{2}-k-2i)}\big)^\vee\\
&\cdot\frac{1}{2}\big(1^{\otimes n}+\delta\cdot e_2^{\otimes n}\big)\cdot e^{(n)}_{\ell,\mathbb{H}}e^{(n)}_{m,\mathbb{O}}\,,
\end{align*}
$k\in\{n-\ell,\ldots,\frac{n}{2}-1\}$, $\delta\in\{-1,1\}$, $\ell\in\left\{\frac{n}{2}+1,\ldots,m\right\}$,
if $m>\frac{n}{2}$,
form a complete set of $\frac{1}{4}(2m-n+2)^2$
primitive orthogonal idempotents in $\Sym^n\mathbb{H}\cdot e^{(n)}_{m,\mathbb{O}}$.
\end{enumerate}
\end{theorem}

\begin{proof}
Let $m$ be an integer between $\left\lceil \frac{n}{2} \right\rceil$ and $n$.
By Proposition~\ref{decompositionProp}(f), applied to $\mathcal{C}=\mathbb{H}$,
\begin{equation*}
\Sym^n\mathbb{H}=\bigoplus_{\ell=\left\lceil \frac{n}{2} \right\rceil}^n\Sym^n\mathbb{H}\cdot e_{\ell,\mathbb{H}}^{(n)}.
\end{equation*}
Hence,
\begin{equation*}
\Sym^n\mathbb{H}\cdot e_{m,\mathbb{O}}^{(n)}=\bigoplus_{\ell=\left\lceil \frac{n}{2} \right\rceil}^n
\Sym^n\mathbb{H}\cdot e_{\ell,\mathbb{H}}^{(n)}e_{m,\mathbb{O}}^{(n)}.
\end{equation*}
By Corollary~\ref{MainLemma}
\begin{equation*}
\Sym^n\mathbb{H}\cdot e_{m,\mathbb{O}}^{(n)}=\bigoplus_{\ell=\left\lceil \frac{n}{2} \right\rceil}^m
\Sym^n\mathbb{H}\cdot e_{\ell,\mathbb{H}}^{(n)}e_{m,\mathbb{O}}^{(n)}.
\end{equation*}
Moreover, for each $\ell$ between $\left\lceil \frac{n}{2} \right\rceil$ and $m$,
the map $r\cdot e_{\ell,\mathbb{H}}^{(n)}\mapsto r\cdot e_{\ell,\mathbb{H}}^{(n)}e_{m,\mathbb{O}}^{(n)}$,
for $r\in\Sym^n\mathbb{H}$, gives an isomorphism
$\Sym^n\mathbb{H}\cdot e_{\ell,\mathbb{H}}^{(n)}\cong\Sym^n\mathbb{H}\cdot e_{\ell,\mathbb{H}}^{(n)}e_{m,\mathbb{O}}^{(n)}$
of $\mathbb{R}$-algebras.

Since the number of elements in a complete set of primitive orthogonal idempotents in 
$\Sym^n\mathbb{H}_\mathbb{C}\cdot e_{\ell,\mathbb{H}}^{(n)}$ is $2\ell-n+1$, for each $\ell$ between $\left\lceil \frac{n}{2} \right\rceil$ and $m$,
the number of elements in a complete set of primitive orthogonal idempotents in 
$\Sym^n\mathbb{H}_\mathbb{C}\cdot e_{m,\mathbb{O}}^{(n)}$ is $\sum_{\ell=\left\lceil \frac{n}{2} \right\rceil}^m(2\ell-n+1)$.
Note that if $n$ is even, then
\begin{equation*}
\begin{aligned}
\sum_{\ell=\frac{n}{2}}^m(2\ell-n+1)&=
2\cdot\frac{m(m+1)}{2}-2\cdot\frac{\left(\frac{n}{2}-1\right)\frac{n}{2}}{2}
+\left(m-\frac{n}{2}+1\right)\cdot(1-n)\\
&=m^2+\frac{n^2}{4}-mn+2m-n+1=\frac{1}{4}(2m-n+2)^2,
\end{aligned}
\end{equation*}
and if $n$ is odd, then
\begin{equation*}
\begin{aligned}
\sum_{\ell=\frac{n+1}{2}}^m(2\ell-n+1)&=
2\cdot\frac{m(m+1)}{2}-2\cdot\frac{\frac{n-1}{2}\frac{n+1}{2}}{2}
+\left(m-\frac{n+1}{2}+1\right)\cdot(1-n)\\
&=m^2+\frac{n^2+3}{4}-mn+2m-n=\frac{1}{4}(2m-n+1)(2m-n+3).
\end{aligned}
\end{equation*}

Thus, by Corollary~\ref{completeCorollary}, the theorem is a result of Theorem~\ref{Theorem1}.
\end{proof}

By Proposition~\ref{decompositionProp}(b), applied to $\mathcal{C}=\mathbb{O}$, the set
$\big\{e_{m,\mathbb{O}}^{(n)}\mid m\in\{\left\lceil \frac{n}{2} \right\rceil,\ldots,n\}\big\}$ 
consists of orthogonal central idempotents of $\Sym^n\mathbb{O}$ that sum to $1^{\otimes n}$.
Note that if $n$ is odd, then
\begin{equation*}
\begin{aligned}
&\sum_{m=\frac{n+1}{2}}^n\frac{1}{4}(2m-n+1)(2m-n+3)=\sum_{i=1}^{\frac{n+1}{2}}i(i+1)\\
&=\frac{1}{3}\frac{n+1}{2}\left(\frac{n+1}{2}+1\right)\left(\frac{n+1}{2}+2\right)
=\frac{1}{24}(n+1)(n+3)(n+5)
\end{aligned}
\end{equation*}
and if $n$ is even, then
\begin{equation*}
\begin{aligned}
&\sum_{m=\frac{n}{2}}^n\frac{1}{4}(2m-n+2)^2=\sum_{i=1}^{\frac{n}{2}+1}i^2\\
&=\frac{1}{6}\left(\frac{n}{2}+1\right)\left(\frac{n}{2}+2\right)(n+3)
=\frac{1}{24}(n+2)(n+3)(n+4).
\end{aligned}
\end{equation*}
Hence, by Corollary~\ref{completeCorollary}, the following result is a corollary of Theorem~\ref{Theorem3}.

\begin{corollary}
\label{Corollary4}
\begin{enumerate}[label={\rm(\alph*)},leftmargin=*]
\item
The set
\[
\left\{
  \binom{n}{k}\bigl(a^{\otimes k}\otimes(a^c)^{\otimes(n-k)}\bigr)^\vee 
  e^{(n)}_{\ell,\mathbb{H}}e^{(n)}_{m,\mathbb{O}}
  \;\middle|\;
  \begin{aligned}
    &k\in\{n-\ell,\ldots,\ell\},\\
    &\ell\in\left\{\left\lceil \frac{n}{2} \right\rceil,\ldots,m\right\},\\
    &m\in\left\{\left\lceil \frac{n}{2} \right\rceil,\ldots,n\right\}
  \end{aligned}
\right\}
\]
is a complete set of $\frac{1}{24}(n+1)(n+3)(n+5)$ (resp.\ $\frac{1}{24}(n+2)(n+3)(n+4)$)
primitive orthogonal idempotents in 
$\big(\Sym^n\mathbb{H}\cdot Z(\Sym^n\mathbb{O})\big)\otimes_\mathbb{R}\mathbb{C}$
if $n$ is odd (resp.\ even).

\item
If $n$ is odd, then
\begin{align*}
\binom{n}{k}\frac{n-2k}{2^k}
\sum_{i=0}^{\frac{n-1}{2}-k}
\left(-\frac{1}{2}\right)^i
\frac{\binom{n-2k-i}{i}}{n-2k-i}
\big(1^{\otimes(n-2k-2i)}\otimes\square^{\otimes(k+i)}\big)^\vee\cdot 
e^{(n)}_{\ell,\mathbb{H}}e^{(n)}_{m,\mathbb{O}}
\end{align*}
for
\[
k\in\{n-\ell,\ldots,\tfrac{n-1}{2}\},\quad
\ell\in\{\tfrac{n+1}{2},\ldots,m\},\quad
m\in\{\tfrac{n+1}{2},\ldots,n\},
\]

form a complete set of $\frac{1}{48}(n+1)(n+3)(n+5)$
primitive orthogonal idempotents in
$\Sym^n\mathbb{H}\cdot Z(\Sym^n\mathbb{O})$.

\item
If $n$ is even, then
\[
\left(\frac{1}{2}\right)^{\frac{n}{2}}
\binom{n}{\frac{n}{2}}
\bigl(\square^{\otimes \frac{n}{2}}\bigr)^\vee
\cdot
e_{\frac{n}{2},\mathbb{H}}^{(n)} e_{m,\mathbb{O}}^{(n)},
\qquad
\text{for } m\in\left\{\frac{n}{2},\ldots,n\right\},
\]
and
\[
\left(\frac{1}{2}\right)^{\frac{n}{2}}
\binom{n}{\frac{n}{2}}
\bigl(\square^{\otimes \frac{n}{2}}\bigr)^\vee
\cdot
e_{\ell,\mathbb{H}}^{(n)} e_{m,\mathbb{O}}^{(n)}
\]
together with
\begin{align*}
\binom{n}{k}\frac{\frac{n}{2}-k}{2^k}
\sum_{i=0}^{\left\lfloor \frac{n}{4}-\frac{k}{2} \right\rfloor}
\left(-\frac{1}{4}\right)^i
\frac{\binom{\frac{n}{2}-k-i}{i}}{\frac{n}{2}-k-i}
\big(\square^{\otimes (k+2i)} \otimes (1\otimes1-\square)^{\otimes (\frac{n}{2}-k-2i)}\big)^\vee \\
\cdot \frac{1}{2}\big(1^{\otimes n}+\delta \cdot e_2^{\otimes n}\big)\cdot e_{\ell,\mathbb{H}}^{(n)} e_{m,\mathbb{O}}^{(n)},
\end{align*}
for
$k \in \{n-\ell,\ldots,\tfrac{n}{2}-1\}$, $\delta \in \{-1,1\}, \quad
\ell \in \left\{\tfrac{n}{2}+1,\ldots,m\right\}$ and $m \in \left\{\tfrac{n}{2}+1,\ldots,n\right\}$, form a complete set of $\frac{1}{24}(n+2)(n+3)(n+4)$ primitive orthogonal idempotents in $\Sym^n\mathbb{H} \cdot Z(\Sym^n\mathbb{O})$.
\end{enumerate}
\end{corollary}

\begin{remark}
\label{Qrem}
By \eqref{emn}, \eqref{beta}, \eqref{trianglen}, and \eqref{triangle}:
\[
    e_{\ell,\mathbb{H}}^{(n)}\in\Sym^n\mathbb{Q}[e_1,e_2,e_3]
\]
and
\[
    e_{m,\mathbb{O}}^{(n)}\in\Sym^n\mathbb{Q}[e_1,e_2,e_3,e_4,e_5,e_6,e_7].
\]
Also, by \eqref{a}, \eqref{ac}, $a,a^c\in\mathbb{Q}[\sqrt{-1}][e_1]$,
and by \eqref{square}, $\square\in\Sym^2\mathbb{Q}[e_1]$.

Hence, Theorem~\ref{Theorem1}, Corollary~\ref{Corollary2}, Theorem~\ref{Theorem3}, and Corollary~\ref{Corollary4}
can be formulated for $\mathbb{Q}[\sqrt{-1}]$, $\mathbb{Q}[e_1,e_2,e_3]$, and
$\mathbb{Q}[e_1,e_2,e_3,e_4,e_5,e_6,e_7]$ instead of $\mathbb{C}$, $\mathbb{H}$, and $\mathbb{O}$, respectively.
\end{remark}

\begin{remark}
\label{G2rem}
The algebra of octonions $\mathbb{O}$ has an automorphism group 
isomorphic to the exceptional Lie group $G_2$.
This group acts transitively on the set of all quaternion subalgebras.

Each automorphism extends naturally to an automorphism of the tensor power $\mathbb{O}^{\otimes n}$,
and then, by restriction, to an automorphism of $\Sym^n\mathbb{O}$.

Since
\[
\triangle_\mathbb{O} = \frac{1}{8}\sum_{i=0}^7 e_i \otimes e_i \in \Sym^2\mathbb{O}
\]
does not depend on the choice of basis $(e_i)_{i=0}^7$ of $\mathbb{O}$
(see \cite[paragraph before Lemma 2.1 and equation (2.9)]{Raz21}),
it is fixed by all automorphisms.

Thus, $\triangle_\mathbb{O}^{(n)} \in \Sym^n\mathbb{O}$ is also fixed by all automorphisms.

By construction,
\[
e^{(n)}_{m,\mathbb{O}} \in \mathbb{R}[\triangle_\mathbb{O}^{(n)}]
\]
for each integer $m$ between $\left\lceil \frac{n}{2} \right\rceil$ and $n$.

Hence every element of $Z(\Sym^n\mathbb{O})$ is fixed by all automorphisms.
\end{remark}

\begin{example}
\label{n2m2example}
By Theorem~\ref{Theorem3}(c), applied to $n=2$ and $m=2$, the set
\begin{equation*}
\begin{aligned}
\Big\{&\tau_0:=\square\cdot e_{1,\mathbb{H}}^{(2)}e_{2,\mathbb{O}}^{(2)},\ 
\tau_1:=\square\cdot e_{2,\mathbb{H}}^{(2)}e_{2,\mathbb{O}}^{(2)},\\
&\tau_2:=(1\otimes1-\square)\cdot\frac{1}{2}(1\otimes1+e_2\otimes e_2)\cdot e_{2,\mathbb{H}}^{(2)}e_{2,\mathbb{O}}^{(2)},\\
&\tau_3:=(1\otimes1-\square)\cdot\frac{1}{2}(1\otimes1-e_2\otimes e_2)\cdot e_{2,\mathbb{H}}^{(2)}e_{2,\mathbb{O}}^{(2)}\Big\}
\end{aligned}
\end{equation*}
is a complete set of primitive orthogonal idempotents in $\Sym^2\mathbb{H}\cdot e_{2,\mathbb{O}}^{(2)}$.

By \eqref{e21} and \eqref{e22}, applied to $\mathcal{C}=\mathbb{H}$,
$e_{1,\mathbb{H}}^{(2)}=\triangle_\mathbb{H}$ and $e_{2,\mathbb{H}}^{(2)}=1\otimes1-\triangle_\mathbb{H}$.
Also, by \eqref{stH}, $\square\triangle_\mathbb{H}=\triangle_\mathbb{H}$.
In particular, $(1\otimes1-\square)e_{2,\mathbb{H}}^{(2)}=1\otimes1-\square$. Hence,
since by \eqref{triangle}, for $\mathcal{C}=\mathbb{H}$, and \eqref{square},
\[
\triangle_\mathbb{H} = \frac{1}{4}\left(1\otimes1 + e_1\otimes e_1 + e_2\otimes e_2 + e_3\otimes e_3\right)
\quad\text{and}\quad
\square = \frac{1}{2}\left(1\otimes1 + e_1\otimes e_1\right),
\]
we get
\begin{equation*}
\tau_0
= \square \triangle_\mathbb{H} \cdot e_{2,\mathbb{O}}^{(2)}
= \triangle_\mathbb{H} \cdot e_{2,\mathbb{O}}^{(2)}
= \frac{1}{4}\left(1\otimes1 + e_1\otimes e_1 + e_2\otimes e_2 + e_3\otimes e_3\right)\cdot e_{2,\mathbb{O}}^{(2)},
\end{equation*}
\begin{equation*}
\tau_1
= \square (1\otimes1 - \triangle_\mathbb{H})\cdot e_{2,\mathbb{O}}^{(2)}
= (\square - \triangle_\mathbb{H})\cdot e_{2,\mathbb{O}}^{(2)}
= \frac{1}{4}\left(1\otimes1 + e_1\otimes e_1 - e_2\otimes e_2 - e_3\otimes e_3\right)\cdot e_{2,\mathbb{O}}^{(2)},
\end{equation*}
\begin{equation*}
\tau_2
= (1\otimes1 - \square)\cdot \frac{1}{2}(1\otimes1 + e_2\otimes e_2)\cdot e_{2,\mathbb{O}}^{(2)}
= \frac{1}{4}\left(1\otimes1 - e_1\otimes e_1 + e_2\otimes e_2 - e_3\otimes e_3\right)\cdot e_{2,\mathbb{O}}^{(2)},
\end{equation*}
\begin{equation*}
\tau_3
= (1\otimes1 - \square)\cdot \frac{1}{2}(1\otimes1 - e_2\otimes e_2)\cdot e_{2,\mathbb{O}}^{(2)}
= \frac{1}{4}\left(1\otimes1 - e_1\otimes e_1 - e_2\otimes e_2 + e_3\otimes e_3\right)\cdot e_{2,\mathbb{O}}^{(2)}.
\end{equation*}
\end{example}

\begin{remark}
\label{n2m2rem}
Consider the elements
\begin{equation*}
\begin{aligned}
\rho_1&=\frac{1}{4}(1\otimes1+e_1\otimes e_1+e_2\otimes e_2+e_3\otimes e_3)\cdot e_{2,\mathbb{O}}^{(2)},\\
\rho_2&=\frac{1}{4}(1\otimes1+e_1\otimes e_1+e_4\otimes e_4+e_5\otimes e_5)\cdot e_{2,\mathbb{O}}^{(2)},\\
\rho_3&=\frac{1}{4}(1\otimes1+e_1\otimes e_1+e_6\otimes e_6+e_7\otimes e_7)\cdot e_{2,\mathbb{O}}^{(2)},\\
\rho_4&=\frac{1}{4}(1\otimes1+e_2\otimes e_2+e_4\otimes e_4+e_6\otimes e_6)\cdot e_{2,\mathbb{O}}^{(2)},\\
\rho_5&=\frac{1}{4}(1\otimes1+e_2\otimes e_2+e_5\otimes e_5+e_7\otimes e_7)\cdot e_{2,\mathbb{O}}^{(2)},\\
\rho_6&=\frac{1}{4}(1\otimes1+e_3\otimes e_3+e_4\otimes e_4+e_7\otimes e_7)\cdot e_{2,\mathbb{O}}^{(2)},\\
\rho_7&=\frac{1}{4}(1\otimes1+e_3\otimes e_3+e_5\otimes e_5+e_6\otimes e_6)\cdot e_{2,\mathbb{O}}^{(2)}.
\end{aligned}
\end{equation*}
of $T_2\mathbb{O}=\Sym^2\mathbb{O}\cdot e_{2,\mathbb{O}}^{(2)}$. Their sum
\begin{equation*}
\sum_{i=1}^7 \rho_i
= \frac{1}{4}\Big(7\cdot 1\otimes1 + 3\cdot \sum_{i=1}^7 e_i\otimes e_i\Big)\cdot e_{2,\mathbb{O}}^{(2)}
\stackrel{\eqref{triangle}}{=}
(1\otimes1 + 6\triangle_\mathbb{O})\cdot e_{2,\mathbb{O}}^{(2)}
\stackrel{\mathrm{Prop.\ }\ref{decompositionProp}\,\mathrm{(d)}}{=}
e_{2,\mathbb{O}}^{(2)}.
\end{equation*}
is the unit element of $T_2\mathbb{O}$ and they are orthogonal:
\begin{equation*}
\rho_i\rho_j=\delta_{ij}\rho_i\quad\text{ for }i,j=1,\ldots,7.
\end{equation*}
For example,
\begin{equation*}
\rho_1^2\stackrel{\eqref{triangle}}{=}\big(\triangle_\mathbb{H}\cdot e_{2,\mathbb{O}}^{(2)}\big)^2
=\triangle_\mathbb{H}^2\cdot e_{2,\mathbb{O}}^{(2)}
\stackrel{\eqref{triangleSq}}{=}\triangle_\mathbb{H}\cdot e_{2,\mathbb{O}}^{(2)}=\rho_1
\end{equation*}
and, by the multiplication table \eqref{MultTable},
\begin{equation*}
\begin{aligned}
\rho_1\rho_2
&= \frac{1}{16}\Big(
(1\otimes1+e_1\otimes e_1+e_2\otimes e_2+e_3\otimes e_3)\\
&\quad +(e_1\otimes e_1+1\otimes1+e_3\otimes e_3+e_2\otimes e_2)\\
&\quad +(e_4\otimes e_4+e_5\otimes e_5+e_6\otimes e_6+e_7\otimes e_7)\\
&\quad +(e_5\otimes e_5+e_4\otimes e_4+e_7\otimes e_7+e_6\otimes e_6)
\Big)\cdot e_{2,\mathbb{O}}^{(2)}\\
&\stackrel{\eqref{triangle}}{=}
\triangle_\mathbb{O}\cdot e_{2,\mathbb{O}}^{(2)}
\stackrel{\mathrm{Prop.\ }\ref{decompositionProp}\,\mathrm{(d)}}{=}
0.
\end{aligned}
\end{equation*}
Note that, since
\[
\triangle_\mathbb{O} = \frac{1}{8}\sum_{i=0}^7 e_i\otimes e_i
\quad\text{and}\quad
\triangle_\mathbb{O} \cdot e_{2,\mathbb{O}}^{(2)} = 0,
\]
we have, in the notation of Example~\ref{n2m2example},
\begin{equation*}
\tau_0=\rho_1,\quad
\tau_1=\rho_2+\rho_3,\quad
\tau_2=\rho_4+\rho_5,\quad
\tau_3=\rho_6+\rho_7.
\end{equation*}
Thus, the set $\{\tau_0,\tau_1,\tau_2,\tau_3\}$ is not a maximal set of orthogonal idempotents
in $T_2\mathbb{O}=\Sym^2\mathbb{O}\cdot e_{2,\mathbb{O}}^{(2)}$ that sum to the unit element $e_{2,\mathbb{O}}^{(2)}$ of $T_2\mathbb{O}$.
\end{remark}

\end{document}